\documentclass[11pt]{article}

\usepackage{amsmath,amssymb,amsthm,color}
\usepackage{algorithm,algorithmicx,algpseudocode}
\usepackage{subcaption}
\usepackage[hidelinks,colorlinks=true,linkcolor=blue,citecolor=black]{hyperref}
\usepackage{fancybox,graphicx,bm,float}               

\newtheorem{definition}{Definition}
\newtheorem{assumption}{Assumption}
\newtheorem{fact}{Fact}
\newtheorem{theorem}{Theorem}

\newtheorem{proposition}[theorem]{Proposition}

\newtheorem{remark}{Remark}[section]
\newtheorem{example}{Example}

\newcommand{\R}{\mathbb{R}}

\newcommand{\cM}{\mathcal{M}}
\newcommand{\cA}{\mathcal{A}}

\newcommand{\cS}{\mathcal{S}}
\newcommand{\cC}{\mathcal{C}}

\newcommand{\cY}{\mathcal{Y}}

\newcommand{\cX}{\mathcal{X}}
\newcommand{\cV}{\mathcal{V}}

\newcommand{\bpmat}{\begin{pmatrix}}
\newcommand{\epmat}{\end{pmatrix}}
\newcommand{\Symn}{\mathcal{S}^{n \times n}}

\newcommand{\Tr}{\text{Tr}}
\newcommand{\bx}{\bm{x}}

\newcommand{\by}{\bm{y}}
\newcommand{\bZ}{\bm{Z}}

\newcommand{\bb}{\bm{b}}
\newcommand{\bd}{\bm{d}}

\newcommand{\bz}{\bm{z}}
\newcommand{\bX}{\bm{X}}
\newcommand{\bY}{\bm{Y}}

\newcommand{\bQ}{\bm{Q}}

\newcommand{\bA}{\bm{A}}

\newcommand{\bC}{\bm{C}}
\newcommand{\bP}{\bm{P}}

\newcommand{\blambda}{\bm{\lambda}}
\newcommand{\bLambda}{\bm{\Lambda}}

\newcommand{\diag}{\operatorname{diag}}
\newcommand{\Diag}{\operatorname{Diag}}

\DeclareMathOperator*{\argmin}{argmin}

\newcommand{\Rmnum}[1]{\uppercase\expandafter{\romannumeral #1}} 
\usepackage[raggedright]{titlesec}
\titleformat{\chapter}{\centering\Huge\bfseries}{Chapter \Rmnum{\thechapter} }{1em}{} 

\usepackage{mathrsfs} 
\usepackage{enumerate} 
\graphicspath{{Figures/}}

\usepackage{multirow}
\usepackage{stmaryrd}

\title{ 
Spectrally Constrained Optimization
}

\author{ Casey Garner\thanks{School of Mathematics, University of Minnesota (\href{mailto:garne214@umn.edu}{garne214@umn.edu}, \href{mailto:lerman@umn.edu}{lerman@umn.edu}) }
\hspace{1cm}
Gilad Lerman\footnotemark[1]
\hspace{1cm}
Shuzhong Zhang\thanks{Department of Industrial and Systems Engineering (\href{mailto:zhangs@umn.edu}{zhangs@umn.edu})}}

\date{\today}

\begin{document}
\maketitle

\vspace{-0.2in}
\begin{abstract}
{We investigate how to solve smooth matrix optimization problems with general linear inequality constraints on the eigenvalues of a symmetric matrix. We present solution methods to obtain exact global minima for linear objective functions, i.e., $F(\bX) = \langle \bC, \bX \rangle$, and perform exact projections onto the eigenvalue constraint set. Two first-order algorithms are developed to obtain first-order stationary points for general non-convex objective functions. Both methods are proven to converge sublinearly when the constraint set is convex. Numerical experiments demonstrate the applicability of both the model and the methods. 

\vspace{3mm}

\vspace{3mm}
}

\end{abstract}

\section{Introduction}
Constrained matrix optimization is an essential aspect of machine learning \cite{duchi2011adaptive}, matrix factorization and completion \cite{mat_completion,matrix_completion_with_noise,mat_factor}, semidefinite programming \cite{vandenberghe1996semidefinite}, robust subspace recovery \cite{robust_subspace_recovery_lerman,rsr} and covariance estimation \cite{fan2016overview}. 
There are two types of constraints to consider in matrix optimization: coordinate constraints and spectral constraints. 
Coordinate constraints impose prohibitions on the entries of the matrix, such as the diagonal elements must equal one or the row and column sums must be unity. Spectral constraints force the eigenvalues or singular values to satisfy certain conditions. For example, assuming a matrix is positive semidefinite enforces non-negativity on the eigenvalues. The literature is replete with a diverse array of coordinate constraints, but the variety of spectral constraints is limited. Given the power of constrained matrix optimization, the lack of general spectral constraints, with corresponding theory and algorithmic development, is a knowledge gap worth filling. 

The goal of this paper is to explore spectrally constrained matrix optimization (SCO) and develop theory and algorithms to solve new models with general spectral constraints. We begin with the following spectrally constrained eigenvalue model     
\begin{align}\label{eqn:SCO-Eig}\tag{SCO-Eig}
\min&\; F(\bX) \\
\text{s.t.}&\; \bA \lambda(\bX) \leq \bb \nonumber \\
           &\; \bX \in \Symn \nonumber 
\end{align}
where $\Symn$ denotes the set of real $n$-by-$n$ symmetric matrices, $F: \Symn \rightarrow \R$ is continuously differentiable, $\bA \in \R^{m \times n}$, $\bb \in \R^m$, and $\lambda(\bX) := \left(\lambda_1(\bX), \cdots, \lambda_n(\bX)\right)^\top \in \R^n$ is the vector of eigenvalues of $\bX$ in descending order, i.e., $\lambda_1(\bX) \geq \lambda_2(\bX) \geq \cdots \geq \lambda_n(\bX)$. \eqref{eqn:SCO-Eig} allows general linear inequality constraints on the eigenvalues; therefore, the model encompasses traditional spectral constraints, such as non-negativity, while significantly boosting the modeling power of the practitioner. As an example, 
low-rank conditions are important and popular in matrix optimization \cite{li2023partial,zhu2018global}. The general low-rank model, 
$
\min\{ F(\bX) \;|\;  \text{rank}(\bX) \leq k, \; \bX \in \Symn_+\}\;
$
where $\Symn_+$ denotes the set of real $n$-by-$n$ positive semidefinite matrices, can be written in the form of \eqref{eqn:SCO-Eig} as 
$
\min\left\{ F(\bX) \;|\; \lambda_n(\bX)  \geq  0, \lambda_{k+1}(\bX) \leq 0, \; \bX \in \Symn\right\} 
$
 or approximated as 
$
\min\left\{ F(\bX) \;|\;  \lambda_i(\bX)  \in [0,\delta], i=k+1, \hdots, n, \; \bX \in \Symn\right\},
\;$
  where $\delta \geq 0$ controls the accuracy of the approximation. 

This paper presents the first study of \eqref{eqn:SCO-Eig} and is a departure from the current corpus devoted to matrix optimization. Our approach offers a change of perspective which has somehow escaped scrutiny by the community. 
To demarcate the novelty of our framework, we embark on a brief excursion into the literature.   

\subsection{Literature Review}
Matrix optimization with an emphasis on the spectrum has been a subject of rigorous exploration for decades. These efforts have amassed a sizable body of work known as {\it eigenvalue optimization} \cite{cullum1975minimization,lewis2003mathematics,lewis1996eigenvalue,mengi2014numerical,overton1992large,overton1988minimizing,overton1993optimality,shapiro1995eigenvalue,ying2012distance}. The main focus of these papers is minimizing functions of the eigenvalues or singular values of a constrained parameterized matrix. A general framework often seen in eigenvalue optimization models is,  
\begin{align}\label{eqn:eig_opt_1}
\min&\; f(\blambda(\cA(\bx))) \\ 
\text{s.t.}&\; \bx \in \Omega \subseteq \R^m, \nonumber 
\end{align}
where $\cA:\R^m \rightarrow \Symn$ is a smooth map and $\Omega$ could be a strict subset of $\R^m$. Many of these works focus on a specific form of the objective function in \eqref{eqn:eig_opt_1}, such as minimizing the maximum eigenvalue of a parameterized matrix \cite{overton1992large} or minimizing a weighted-sum of the $k$-largest eigenvalues or singular values \cite{cullum1975minimization,overton1993optimality,ying2012distance}. According to Overton \cite{overton1993optimality}, the work of Cullum, Donath and Wolfe in 1975 seems to be the first study of minimizing the sum of the $k$-largest eigenvalues of a parameterized matrix. In particular, Cullum et al.~investigated minimizing $f_{\kappa}(\bx) = \sum_{i=1}^{\kappa} \lambda_i(\cA(\bx))$ where $\cA(\bx) = \bA_0 + \diag(\bx)$ for a fixed matrix $\bA_0$. In the case of Hermitian matrices, Mengi et al.~\cite{mengi2014numerical} provide a general framework similar to \eqref{eqn:eig_opt_1}, and Kangal et al.~\cite{kangal2018subspace} consider an infinite dimensional problem by minimizing the $k$-th largest eigenvalue of a compact self-adjoint operator. Early overviews of eigenvalue optimization are provided by Lewis and Overton \cite{lewis2003mathematics,lewis1996eigenvalue}. 

The seminal work of Cullum, Donath and Wolfe focused on understanding the non-smoothness of $f_\kappa$, and the study of non-smooth functions composed of eigenvalues and singular values has remained a common direction of inquiry. We see this in the works of Overton \cite{overton1992large,overton1988minimizing,overton1993optimality} and Lewis \cite{hiriart1999clarke,lewis1996derivatives,lewis1999nonsmooth,lewis2001twice} where much discussion surrounds computing the subdifferentials of functions whose arguments are eigenvalues and singular values. 
A particularly important setting where subdifferentials and derivatives are readily available in eigenvalue optimization is that of {\it spectral functions}. Spectral functions are a staple in the eigenvalue optimization literature \cite{beck2017first,lewis1996derivatives,lewis2001twice}. 
The function $F:\Symn \rightarrow \R$ is a spectral function (over $\Symn$) if $F(\bX) = f(\blambda(\bX))$ for some {\it symmetric function} $f:\R^n \rightarrow \R$ where $f$ {\it symmetric} means $f(\bx) = f(\bm{P}\bx)$ for all permutation matrices $\bm{P} \in \R^{n\times n}$. The key feature of spectral functions is they are independent of the order of the eigenvalues, e.g., $F(\bX) = \text{Tr}(\bX) = \sum_{i=1}^{n} \lambda_i(\bX)$ is a spectral function. Example 7.13 in \cite{beck2017first} presents a number of common spectral functions encountered in the literature.   

Studies on how to avoid non-smoothness also exist. Shairo and Fan \cite{shapiro1995eigenvalue} detailed how to avoid some issues of non-smoothness at repeated eigenvalues by proving under certain conditions the set $\left\{ \bx \in \R^m \; | \; \lambda_1(\mathcal{A}(\bx)) = \cdots = \lambda_k(\mathcal{A}(\bx))\right\}$ is a smooth manifold near a point $\bx^*$ where $\lambda_1(\mathcal{A}(\bx^*))$ has multiplicity and $\mathcal{A}$ is a smooth map from $\R^m$ to $\Symn$.

\subsection{Contributions: A New Perspective}
A common thread present in the eigenvalue optimization literature is minimizing objective functions solely dependent on the spectrum of the decision matrix. As a result, these papers fixate on exploring the non-smoothness of these objectives. This has led to beautiful mathematics and robust application; however, the success of these models has simultaneously limited the scope of the investigation. Constrained matrix optimization revolves around coordinate and spectral constraints, and the current eigenvalue optimization literature has not expanded our understanding of general spectral constraints. 

Our work on \eqref{eqn:SCO-Eig} begins a novel study to understand eigenvalues being present functionally in the constraint set. Ours is a perspective shift, instead of focusing on eigenvalues in the objective, we focus on eigenvalues in the constraint, and, to our surprise, this perspective seems uncommon. Almost no papers consider general functional constraints on the eigenvalues, e.g., $g(\lambda(\bX)) \leq 0$, and those that do restrict themselves to spectral functions. In the very recent work \cite{li2023partial}, which appeared while we were finishing this paper, the authors allow spectral functions in the constraint set they consider; however, they do not go beyond this paradigm and they further require convexity of the spectral functions. We do not assume the constraint set in \eqref{eqn:SCO-Eig} is composed of spectral functions, and, unlike \cite{li2023partial}, we develop approaches to solve our model without convexification. Another work which considers general functional eigenvalue constraints is \cite{jourani2005error}. In this paper, the authors investigate error bound conditions for constraint sets of the form $\left\{\bX \in \Symn \; | \; f(\blambda(\bX)) \leq 0 \right\}$ where $f:\R^n \rightarrow \R$, but they did not utilize their results to investigate matrix optimization.  

The main contribution of this work is the first study of a matrix optimization model with general linear inequality constraints on the eigenvalues. We prove in the case of a linear objective function that a global minimum of \eqref{eqn:SCO-Eig} can be computed regardless of the non-convexity of the constraint set, and we prove how to perform optimal projections onto the constraint set. With these results, we develop two first-order algorithms to compute stationary points to \eqref{eqn:SCO-Eig} and provide a proof of their sublinear convergence rate. Numerical experiments demonstrate the applicability of our procedures. 

\subsection{Organization} 
The organization of the paper is as follows. 
Section \ref{sec:constraint_set} presents and proves some facts about the eigenvalue constraint set, such as its connectedness.
Section \ref{sec:opt_conditions} provides a discussion of necessary optimality conditions for \eqref{eqn:SCO-Eig} and an approximated version of the model which avoids the non-smoothness associated with repeated eigenvalues. Sections \ref{sec:lin_obj} and \ref{sec:projection} demonstrate we can solve essential instantiations of \eqref{eqn:SCO-Eig} which form crucial steps in classical optimization algorithms. Namely, Section \ref{sec:lin_obj} proves solving \eqref{eqn:SCO-Eig} with a linear objective function can be done exactly, regardless of the non-convexity of the constraint, and only requires solving a linear program and performing a spectral decomposition. Section \ref{sec:projection} solves the projection problem onto the eigenvalue constraint set. From these results, in Section \ref{sec:solvers} we present a projected gradient method and a Frank-Wolfe algorithm which obtain first-order $\epsilon$-stationary points to \eqref{eqn:SCO-Eig} when the constraint set is convex and the objective function is smooth though possibly non-convex. Section \ref{sec:experiments} displays the results of numerical experimentation; Section \ref{sec:precond_example} applies both methods to the question of determining a preconditioning matrix for solving linear systems, and Section \ref{sec:quadratic_equations} investigates the solving of systems of quadratic equations with our projected gradient method. We conclude the paper in Section \ref{sec:conclusion} with directions for future inquiries.

In this paper, scalars, vectors and matrices are denoted by lower case, bold lower case and bold upper case letters respectively, e.g., $x$, $\bm{y}$ and $\bm{Z}$. Given the positive integer $k$, $[k]:=\left\{1,\hdots, k\right\}$. If $\bx \in \R^n$, $\Diag(\bx)$ denotes the diagonal matrix whose diagonal is $\bx$. $\cS^{n\times n}$, $\cS^{n\times n}_+$, and $\cS^{n\times n}_{++}$ denote the set of real $n$-by-$n$ symmetric, positive semidefinite, and positive definite matrices respectively. Let $\mathcal{O}(n):=\left\{ \bX \in \R^{n\times n} \; | \; \bX^\top \bX = \bm{I} \right\}$ denote the set of orthogonal matrices where $\bm{I}$ is the identity matrix.   

\section{The Eigenvalue Constraint Set}\label{sec:constraint_set}
The paradigm shift we are pursuing rests on the eigenvalues being functionally present in the constraints. Due to the limited attention received by such settings, we begin by proving some general facts about the constraint set in \eqref{eqn:SCO-Eig}, which we denote as
\begin{equation}\label{eqn:eig_constraint_set}
\mathcal{S}_{\blambda}(\bA,\bb) := \left\{ \bX \in \Symn \; | \; \bA \blambda(\bX) \leq \bb\right\}. \end{equation}
We shall desire to take advantage of writing the constraint without ordering the eigenvalues. To this end, let $\blambda_{\bX} \in \R^n$ be the unordered vector of eigenvalues for $\bX\in \Symn$; we define the matrix $\bm{D}_n \in \R^{(n-1)\times n}$ to enforce the descending order of the eigenvalues, i.e., the $i$-th row of $\bm{D}_n$ is $\bm{e}_{i+1} - \bm{e}_{i}$ where $\bm{e}_i \in \R^n$ has a one in the $i$-th place and zeros elsewhere. Thus, an equivalent manner of writing \eqref{eqn:eig_constraint_set} is 
$
\mathcal{S}_{\blambda}(\bA,\bb) = \left\{ \bX \in \Symn \; | \; \bA \blambda_{\bX} \leq \bb,\; \bm{D}_n \blambda_{\bX} \leq \bm{0} \right\}.\;  
$
Our first fact demonstrates the unsurprising but key truth that $\mathcal{S}_{\blambda}(\bA,\bb)$ has the potential to be convex and non-convex as a function of the problem's data. 
\begin{fact}\label{fact:one}
$\mathcal{S}_{\blambda}(\bA,\bb)$ can be both convex and non-convex dependent on $\bA \in \R^{m\times n}$ and $\bb \in \R^m$.
\end{fact}
\begin{proof}
If $\bA = -\bm{I}$ and $\bb = \bm{0}$, then $\mathcal{S}_{\blambda}(\bA,\bb) = \mathcal{S}^{n\times n}_{+}$ which is convex. Let $\bA \in \R^{2 \times 2}$ and $\bb \in \R^2$ enforce the constraint $\bX \in \mathcal{S}^{2 \times 2}$ with $\lambda_1(\bX) \geq 3$ and $\lambda_2(\bX) \leq 1$. One can easily produce examples where convex combinations of matrices satisfying these constraints fail these conditions. For example, the matrices 
\[
\bX_1 = \begin{pmatrix} 35 & 15 \\ 15 & 6 \end{pmatrix} \;\; \text{ and }\;\; \bX_2 = \begin{pmatrix} 4 & 17 \\ 17 & 63 \end{pmatrix}
\]
satisfy the eigenvalue constraints; however, $\bY:=\frac{1}{2}\left( \bX_1 + \bX_2 \right)$ fails the condition $\lambda_2(\bY) \leq 1$.  
\end{proof}

We further explore the non-convexity of $\mathcal{S}_{\blambda}(\bA,\bb)$ with the following example. 

\begin{example}\label{example:1}
Let 
$\mathcal{S}_{\blambda}(\bA,\bb) = \left\{\bX \in \mathcal{S}^{2\times 2} \; | \; \blambda_1(\bX) \in [3,5],\; \blambda_2(\bX) \in [0,2]   \right\}$. This set is non-convex. To visualize this, we take slices of the constraint set. Let
\[
\bX = \begin{pmatrix} x_1 & x_3 \\ x_3 & x_2\end{pmatrix} \;\; \text{ with } \; x_1, x_2, x_3 \in \R. 
\]
Fixing $x_3$ for different values, we plot in Figure \ref{fig:1} the values of $x_1$ and $x_2$ such that $\bX \in \mathcal{S}_{\blambda}(\bA,\bb)$. The three subplots in Figure \ref{fig:1} clearly display the non-convexity of the set. When $x_3=0$, the disjoint interval constraints become evident in the two square regions of the third subplot.  

\begin{figure}[H]
\includegraphics[width=\textwidth]{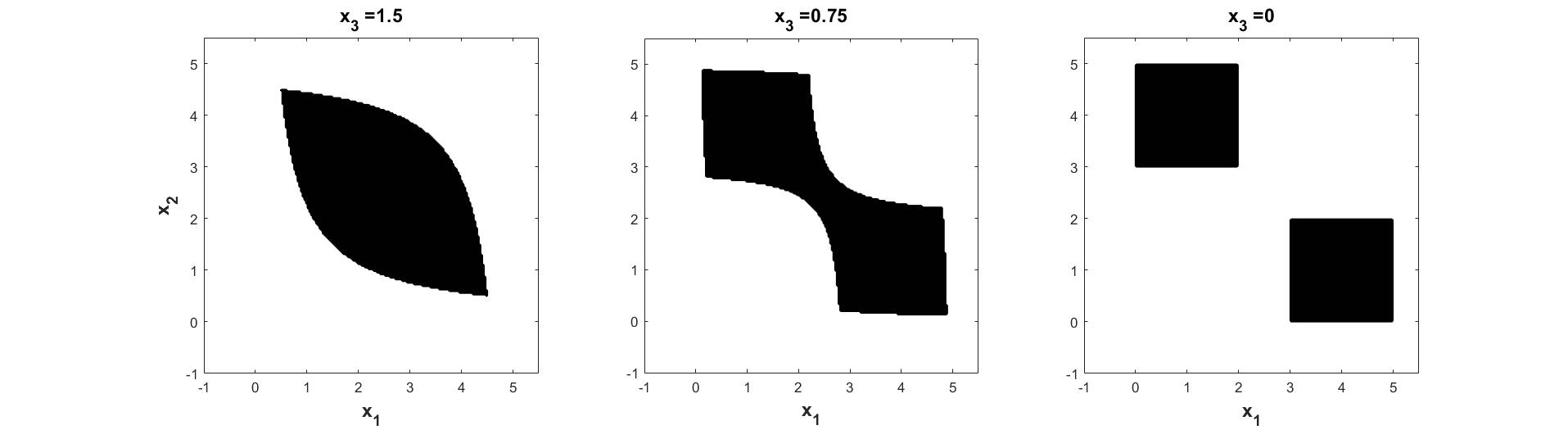}
\caption{Demonstration of $\mathcal{S}_{\blambda}(\bA,\bb)$ in Example~\ref{example:1}. Each subplot shows the black regions of coordinates $x_1$ and $x_2$ of $\bX \in \mathcal{S}_{\blambda}(\bA,\bb)$ for a fixed value of $x_3$, which is specified above the subplot.\label{fig:1}}
\end{figure}
\end{example} 

Though the constraint set in \eqref{eqn:SCO-Eig} may be non-convex, the feasible region, however, is always connected. 
\begin{fact}
$\mathcal{S}_{\blambda}(\bA,\bb)$ is a connected subset of $\Symn$. 
\end{fact}
\begin{proof}
Let $\bX_1 = \bQ_1 \bLambda_1 \bQ_1^\top$ and $\bX_2 = \bQ_2 \bLambda_2 \bQ_2^\top$ be arbitrary elements of $\mathcal{S}_{\blambda}(\bA,\bb)$. 
Without loss of generality, we may assume $\bQ_1, \bQ_2 \in SO(n):=\left\{ X \in \mathcal{O}(n) \; | \; \det(\bX)=1 \right\}$. 
Moreover, we may assume without loss of generality the diagonal elements of $\bLambda_1$ and $\bLambda_2$ are in descending order. Using the fact $SO(n)$ is path connected\footnote{A short proof of this result is presented at: \url{https://www.math.tamu.edu/~rojas/son.pdf}}, there exists a continuous map $\bm{G}:[0,1]\rightarrow SO(n)$ such that $\bm{G}(0)=\bQ_1$ and $\bm{G}(1) = \bQ_2$. Thus, we can define the continuous parameterization $\bX:[0,1]\rightarrow \Symn$ such that 
\[
\bX(t) = \bm{G}(t) \left( \bLambda_2 + (1-t) \left(\bLambda_1 - \bLambda_2 \right)\right)\bm{G}(t)^\top. 
\]
By construction, $\bX(0)=\bX_1$ and $\bX(1)=\bX_2$, and the convexity of $\left\{ \blambda \in \R^n \; | \; \bA \blambda \leq \bb, \; \bm{D}_n \blambda \leq \bm{0} \right\}$ ensures $\bX(t) \in \mathcal{S}_{\blambda}(\bA,\bb)$ for all $t\in[0,1]$. 
\end{proof}

\begin{remark}
Figure \ref{fig:1}, especially the third subplot in Figure \ref{fig:1}, might seem to suggest a counter-example refuting the connectedness of $\mathcal{S}_{\blambda}(\bA,\bb)$. However, this is not the case. The stills in Figure \ref{fig:1} represent slices of the feasible region and paths connecting points in $\mathcal{S}_{\blambda}(\bA,\bb)$ are not restricted to slices.   
\end{remark}



%

The capacity for non-convexity in the constraint is the trade-off incurred by removing the eigenvalues from the objective. Potential non-smoothness of the objective function associated with the eigenvalues has been replaced with potential non-convexity associated with the eigenvalues in the constraint. Constrained problems over non-convex sets are exceptionally challenging; therefore, the cases where $\mathcal{S}_{\blambda}(\bA,\bb)$ is convex are of interest. The next theorem provides a verifiable condition which ensures the convexity of the constraint and proves the convexity of common eigenvalue constraints such as $\mathcal{S}^{n \times n}_{+}$ and condition number constraints \cite{tanaka2014positive,won2013condition}. 

\begin{theorem}\label{thm:convex_constraint}
If each row of $\bA$ is an element of $\R^n_{\geq} := \left\{ \bx \in \R^n \; | \; x_1 \geq \hdots \geq x_n \right\}$, $\mathcal{S}_{\blambda}(\bA,\bb)$ is convex. 
\end{theorem}
\begin{proof}
Assume each row of $\bA$ is an element of $\R^n_{\geq}$. Define $f_i(\bX):= \sum_{j=1}^n A_{ij} \lambda_j(\bX)$ for $i\in [m]$. Since each $f_i$ is convex (see Section \ref{sec:appen_1}), the level sets of each $f_i$ are convex.   
The result then follows because $\mathcal{S}_{\blambda}(\bA,\bb)$ is the intersection of the level sets of the $f_i$'s. 
\end{proof}

\section{General Optimality Conditions}\label{sec:opt_conditions}
The non-smoothness of the eigenvalues necessitates a discussion of concepts from non-smooth analysis. Consider an open set $\Omega$ of $\R^n$ and $f:\Omega \rightarrow \R$ locally Lipschitz on $\Omega$, i.e., given $\bx \in \Omega$ there exists $L_{\bx}$ and $\delta_{\bx} > 0$ such that if $\|\by - \bx \| \leq \delta_{\bx}$ then 
$|f(\bx) - f(\by) | \leq L_{\bx} \| \bx - \by\|$. The ({\it radial}) {\it directional derivative} of $f$ at $\bx$ in the direction of $\bd \in \R^n$ is
\[
f'(\bx; \bd) := \lim_{t \rightarrow 0_+} \frac{f(\bx + t \bd) - f(\bx)}{t},
\]
when the limit exists. A relaxed version of the directional derivative is the {\it Clarke directional derivative}.
The {\it Clarke directional derivative} of $f$ at $\bx$ in the direction of $\bd$ is
\[
f^C(\bx; \bd):= \limsup_{(t,\by) \rightarrow (0_+, \bx)} \frac{f(\by + t \bd) - f(\by)}{t}.
\]
Unlike $\bd \mapsto f'(\bx;\bd)$, $\bd\mapsto f^C(\bx;\bd)$ is guaranteed to be a finite convex function for all $\bx \in \Omega$ (see Section 5.1 in \cite{penot2013calculus}). The {\it Clarke subdifferential} of $f$ at $\bx$ is
\[
\partial_C f(\bx) := \left\{ \bm{s} \in \R^n \; | \; \langle \bm{s}, \bm{d} \rangle \leq f^C(\bx; \bd), \; \forall \bd \in \R^n \right\}. 
\]
The Clarke subdifferential is a non-empty, compact and convex set, and these properties make it a frequently utilized generalization of differentiation for non-smooth functions. The definitions presented here for Clarke subdifferentials can be found in Section 2 of \cite{hiriart1999clarke}. It is often necessary to assume a function is {\it regular} at a point for certain results to hold. Different definitions of regularity exist, but in this section we say $f$ is {\it regular} at $\bx \in \Omega$ if 
\[
f^C(\bx;\bd) = \liminf_{(t,\bm{v}) \rightarrow (0_+, \bd)} \frac{f(\bx+t\bm{v}) - f(\bx)}{t}
\]
for all $\bd \in \R^n$ (see Definition 5.46 and Proposition 4.3 in \cite{penot2013calculus}).
This notion of regularity focuses on the equality of two slightly different generalizations of the directional derivative of $f$. One deals with perturbing the point $\bx$ while the other perturbs the direction $\bd$. Convex functions are regular at all points in their domain and if $f$ is continuously differentiable at $\bx$ then it is also regular at  $\bx$. 
A thorough discourse on Clarke subdifferentials and regularity is presented in \cite{penot2013calculus}. 
The interested reader should consult this text and the references therein for further details. 
We now present general necessary conditions for optimal solutions to \eqref{eqn:SCO-Eig}. 

\begin{theorem}\label{thm:gen_opt_condition}
Assume $F:\Symn \rightarrow \R$ is continuously differentiable and $g_i(\bX):= \bm{a}_i^\top \blambda(\bX) - b_i$ for $i\in [p]$. Let $\bX^*$ be a local minimizer of $\min\left\{F(\bX) \; | \; g_i(\bX) \leq 0, \; i\in [p], \; \bX \in \Symn\right\}$ and $\mathcal{I}(\bX^*):= \left\{i\in [p] : g_i(\bX^*) = 0\right\}$. Suppose
\[
\bm{t} \in \R^p_+, \; t_i = 0 \; \forall i \in [p] \setminus \mathcal{I}(\bX^*), \; 0 \in \partial_C g_1(\bX^*) + \hdots + \partial_C g_p(\bX^*) \implies \bm{t} = \bm{0}
\]
and each $g_i$ is regular at $\bX^*$, then there exist multipliers $\mu \in \R^{p}_+$ such that, $\mu_i g_i(\bX^*) = 0$ for all $i$ and, 
\begin{multline}\label{eqn:opt_eq_1}
 0 \in  \nabla F(\bX^*) + \mu_1 \sum_{i=1}^{n} {a}_{1i} \text{conv}\left\{ \bm{v} \bm{v}^\top \; | \; \bm{v} \in \mathbb{E}_i(\bX^*), \; \|\bm{v}\|=1\right\} \\  + \hdots + \mu_p \sum_{i=1}^{n} {a}_{pi} \text{conv}\left\{ \bm{v} \bm{v}^\top \; | \; \bm{v} \in \mathbb{E}_i(\bX^*), \; \|\bm{v}\|=1\right\} 
\end{multline}
where $\mathbb{E}_i(\bX^*)$ denotes the eigenspace of $\bX^*$ corresponding to the $i$-th largest eigenvalue of $\bX^*$ and $\text{conv}(S)$ is the convex hull of a given set $S$. 
\end{theorem}




\begin{proof}
Using the framework of Corollary 5.54 in \cite{penot2013calculus}, let $f = F$ and $g_i$ be as defined. Since $\blambda(\cdot)$ is globally Lipschitz continuous, $g_i$ is locally Lipschitz for all $\bX \in \Symn$. By the assumption $F$ is continuously differentiable, it follows $\partial_C F(\bX) = \{\nabla F(\bX)\}$ for all $\bX$. The proof is completed by computing the Clarke subdifferential of the $g_i$'s. Since $g_i$ is regular at $\bX$,
\begin{align}
\partial_C g_i(\bX) &= \sum_{j=1}^{n} \partial_C \left( {a}_{ij} \blambda_j(\cdot)\right)(\bX) \nonumber \\
&= \sum_{j=1}^{n} {a}_{ij} \partial_C \blambda_j(\bX) \nonumber \\
&=\sum_{j=1}^{n} {a}_{ij} \text{conv}\left\{ \bm{v} \bm{v}^\top \;| \; \bm{v} \in \mathbb{E}_j(\bX), \; \|\bm{v}\|=1 \right\} 
\end{align}
where the first and second equalities follow by Theorem 5.51 and Proposition 5.9 in \cite{penot2013calculus} respectively. The last equality is due to Theorem 5.3 in \cite{hiriart1999clarke}. 
\end{proof}

Theorem \ref{thm:gen_opt_condition} provides general optimality conditions for
\eqref{eqn:SCO-Eig}. Note,  
if $g_i(\bX) = \bm{a}_i^\top \lambda(\bX) - b_i$ for $i\in [m]$ where $\bm{a}_i$ is the $i$-th row of $\bA$, then $\mathcal{S}_{\lambda}(\bA,\bb) = \left\{ \bX \in \Symn \; | \; g_i(\bX) \leq 0, \; i\in [m]\right\}$. 
The assumptions required for Theorem \ref{thm:gen_opt_condition} to hold are substantial including: a constraint qualification, regularity and non-trivial convex hulls of eigenspaces; however, if the local minimizer has unique eigenvalues, the necessary conditions simplify immensely.  
%
\begin{theorem}\label{thm:easy_NC}
Assume $F:\Symn \rightarrow \R$ is continuously differentiable. Define $g_i(\bX):= \bm{a}_i^\top \blambda(\bX) - b_i$ for $i\in [p]$. Let $\bX^*$ be a local minimizer of $\min_{\bX \in \Symn}\left\{F(\bX) \; | \; g_i(\bX) \leq 0, \; i\in [p]\right\}$ with no repeated eigenvalues. If $\{ \bm{a}_i \; | \; i \in \mathcal{I}(\bX^*)\}$ is a linearly independent set, then there exist multipliers $\mu \in \R^p_+$ such that, 
\[
\nabla F(\bX^*) + \sum_{i=1}^{p} \mu_i^* \bm{V}^* \text{Diag}(\bm{a}_i) (\bm{V}^*)^\top = \bm{0},
\]
and $\mu_i^* g_i(\bX^*) = 0$ for all $i$ where $\bX^* = \bm{V}^* \text{Diag}(\blambda(\bX^*)) (\bm{V}^*)^\top$. 
\end{theorem}

\begin{proof}
Since $\bX^*$ has no repeated eigenvalues,
$g_i$ is regular at $\bX^*$. The uniqueness of the eigenvalues of $\bX^*$ imply there is a neighborhood about $\bX^*$ such that $\lambda_i(\cdot)$ is continuously differentiable with a continuously varying associated eigenvector (Theorem 3.1.1 of \cite{Ortega_Num_Analysis}). 
Thus, each $g_i$ is continuously differentiable at $\bX^*$ and therefore regular at $\bX^*$.
The calculation of the Clarke subdifferential of $g_i$ follows from the fact each eigenspace of $\bX^*$ contains a single element, 
$\partial_C g_i(\bX^*) = \left\{ \bm{V}^* \text{Diag}(\bm{a}_i) (\bm{V^*})^\top\right\}$ where $\bX^* = \bm{V}^* \text{Diag}(\blambda(\bX^*))(\bm{V^*})^\top$ is the eigendecomposition of $\bX^*$. 
Hence, 
\begin{equation}\label{eqn:help_conqual}
t_1 \partial_C g_1(\bX) + \hdots + t_p \partial_C g_p(\bX) = \sum_{i=1}^{p} t_i \bm{V}^* \text{Diag}(\bm{a}_i) (\bm{V^*})^\top = \bm{V}^* \text{Diag}(\sum_{i=1}^{p} t_i \bm{a}_i) (\bm{V^*})^\top. 
\end{equation}
To complete the proof, we prove the constraint qualification condition for Corollary 5.54 in \cite{penot2013calculus} holds given our assumption on $\{ \bm{a}_i \; | \; i \in \mathcal{I}(\bX^*)\}$. Without loss of generality, we assume the first $r$ constraints are active. Then the constraint qualification will hold provided 
$
\bm{0} = t_1 \partial_C g_1(\bX) + \hdots + t_r \partial_C g_r(\bX) \;
$
if and only if $t_1=\hdots=t_r=0$. From \eqref{eqn:help_conqual}, we see this is equivalent to 
$
\bm{0} = \text{Diag}(t_1 \bm{a}_1 + \hdots + t_r \bm{a}_r)\;
$
if and only if $t_1 = \hdots = t_r=0$ which follows from the assumptions. 
\end{proof}

Comparing the necessary conditions with and without repeated eigenvalues inclines us to prefer the latter. 
The eigenspace dependence, the required regularity, and the constraint qualification limit the utility of Theorem  \ref{thm:gen_opt_condition} while such difficulties dissipate when the local minimizer has no repeated eigenvalues. In some cases, the constraint set ensures uniqueness; however, many constraints have feasible matrices with repeated eigenvalues. This seems unavoidable, but we now prove it is possible to approximate the original model and remove all feasible solutions with repeated eigenvalues. Letting $\epsilon > 0 $, we define the approximated \eqref{eqn:SCO-Eig} model, 
\begin{align}\label{eqn:approx_SCO-Eig}\tag{SCO-Eig-$\epsilon$}
\min&\; F(\bX) \nonumber \\
\text{s.t.}&\; \bA \lambda(\bX) \leq \bb \nonumber \\
           &\;\lambda_{i+1}(\bX) \leq \lambda_{i}(\bX) - \epsilon, \; i=1,\hdots, n-1 \nonumber \\ 
           &\; \bX \in \Symn. \nonumber 
\end{align}
The next result shows that 
it is possible to bound the gap in the optimal values of the two models as a function of $\epsilon$
given certain assumptions.
\begin{theorem}\label{thm:M_eps_AND_M}
Assume $F$ is Lipschitz continuous with constant $L>0$, the interior of $\mathcal{S}_{\blambda}(\bA,\bb)$ is non-empty, and $[\bA^\top | \bm{D}_n^\top] \in \R^{n \times (m+n-1)}$ has full rank. Let $\bX^*$ be a global minimizer of \eqref{eqn:SCO-Eig}. Then there exists $\epsilon_0 > 0$ such the constraint set of \eqref{eqn:approx_SCO-Eig} is non-empty for all $\epsilon \in [0,\epsilon_0]$, and, letting $\bX_{\epsilon}^*$ be a global minimizer of \eqref{eqn:approx_SCO-Eig}, we have for all $\epsilon \in [0,\epsilon_0]$  
\[
| F(\bX^*) - F(\bX_{\epsilon}^*)| \leq L\cdot\chi([\bA^\top | \bm{D}_n^\top])\cdot (\sqrt{n-1})\epsilon,
\]
where $\chi(\bm{Z}) := \left\{ \|\bm{Z}_\mathcal{I}^{-1}\|_2 \; | \; \mathcal{I} \subset [m+n-1],\; |\mathcal{I}|=n,\; \bm{Z}_{\mathcal{I}} \text{ non-singular} \right\}$\footnote{This constant was first presented by Dikin \cite{dikin1967iterative}. Equivalent definitions, as the one stated, were  proven in \cite{vavasis1996primal,zhang2000global}.} with $\bm{Z}_{\mathcal{I}}$ being the matrix formed from the columns of $\bm{Z}$ whose indices belong in $\mathcal{I}$.
\end{theorem}
\begin{proof}
By our assumptions, there exists $\bm{\bar{X}}\in \Symn$ such that $\bA \blambda(\bm{\bar{X}}) < \bb$ and $\bm{D}_n\blambda(\bm{\bar{X}})<\bm{0}$. Therefore, 
there exists ${\epsilon_0}>0$ such that  
$\bA \blambda(\bm{\bar{X}}) \leq \bb$ and $\bm{D}_n\blambda(\bm{\bar{X}})\leq -\epsilon_0 \bm{e}$ which implies the constraint set of \eqref{eqn:approx_SCO-Eig} is non-empty for all $\epsilon \in [0,\epsilon_0]$. 

Let $\bX^* = \bm{V}^* \text{Diag}(\blambda^*)(\bm{V}^*)^\top$ be a spectral decomposition of $\bX^*$ and $P_{\mathcal{S}_\epsilon}(\blambda^*)$ be the projection of $\blambda^*$ onto $\mathcal{S}_\epsilon:=\{\bx \in \R^n \; | \; \bA\bx \leq \bb, \; \bm{D}_n\bx \leq -\epsilon \bm{e}\}$. By the Lipschitz continuity of $F$, for any $\epsilon \in [0,\epsilon_0]$
\begin{align}\label{eqn:diff_F}
| F(\bX^*) - F(\bX_{\epsilon}^*)| &= | F(\bm{V}^* \text{Diag}(\blambda^*)(\bm{V}^*)^\top) - F(\bX_\epsilon^*)| \nonumber\\
                                &\leq | F(\bm{V}^* \text{Diag}(\blambda^*)(\bm{V}^*)^\top) - F(\bm{V}^* \text{Diag}(P_{S_\epsilon}(\blambda^*))(\bm{V}^*)^\top)| \nonumber\\
                                            &\leq L \|\bm{V}^*\text{Diag}(\blambda^*)(\bm{V}^*)^\top - \bm{V}^*\text{Diag}(P_{\mathcal{S}_\epsilon}(\blambda^*))(\bm{V}^*)^\top  \|_F, \nonumber \\ 
                                            &=L\| \blambda^* - P_{\mathcal{S}_\epsilon}(\blambda^*)\|_2. 
\end{align}
Using Hoffman's error bound for linear systems, we bound the distance between $\blambda^*$ and $ P_{\mathcal{S}_\epsilon}(\blambda^*)$. By Theorem 3.6 in \cite{zhang2000global}, for all $\bm{z} \in \R^n$ 
\begin{equation}\label{eqn:hoff_bound}
\| \bm{z} - P_{\mathcal{S}_\epsilon}(\bm{z})\|_2 \leq \chi([\bA^\top|\bm{D}_n^\top])\cdot \bigg\| \left( [\bA^\top|\bm{D}_n^\top]^\top \bm{z} - [\bb^\top | -\epsilon \bm{e}^\top]^\top \right)_+ \bigg\|_2 
\end{equation}
where $(\bm{y})_+ := ((y_1)_+, \hdots, (y_n)_+)^\top$ for $\bm{y}\in \R^n$ with $(a)_+ := \max(0,a)$ for $a\in \R$ and for $\bm{Z} \in \R^{m \times n}$ with full rank 
\begin{equation}\label{eqn:chi_def}
\chi(\bm{Z}) := \left\{ \|\bm{Z}_\mathcal{I}^{-1}\|_2 \; | \; |\mathcal{I}|=m,\; \bm{Z}_{\mathcal{I}} \text{ is non-singular} \right\},
\end{equation}
where $\bm{Z}_{\mathcal{I}}$ denotes the submatrix matrix of $\bm{Z}$ composed of the columns of $\bm{Z}$ in the index set $\mathcal{I} \subset [n]$ \cite{zhang2000global}. Since $\blambda^*$ is contained in $\{\bx \in \R^n \; | \; \bA\bx \leq \bb, \; \bm{D}_n\bx \leq \bm{0}\}$, 
\begin{equation}\label{eqn:bound_lambda}
\bigg\| \left( [\bA^\top|\bm{D}_n^\top]^\top \blambda^* - [\bb^\top | -\epsilon \bm{e}^\top]^\top \right)_+ \bigg\|_2  =  \| (\bm{D}_n \bm{\lambda}^* + \epsilon \bm{e})_+\|_2 \leq \| \epsilon \bm{e}\|_2 = \epsilon \sqrt{n-1}.  
\end{equation}
Combining \eqref{eqn:diff_F}, \eqref{eqn:hoff_bound} and \eqref{eqn:bound_lambda} concludes the argument. 
\end{proof}

\section{Solving SCO-Eig with Linear Objective Functions}\label{sec:lin_obj} 
The simplest, non-trivial objective function to consider for \eqref{eqn:SCO-Eig} is a linear objective function, 
\begin{align}\label{eqn:main_linear_model}
\min&\;  \langle \bm{C}, \bX\rangle \\
\text{s.t.}&\; \bA \blambda(\bX) \leq \bb \nonumber \\
           &\; {\color{black}\bX \in \Symn}. \nonumber 
\end{align}
As we saw in Section \ref{sec:constraint_set}, the constraint set $\mathcal{S}_{\lambda}(\bA,\bb)$ can be highly non-convex. Therefore, one may not expect to guarantee a global minimizer; however, Theorem \ref{thm:gen_lin_obj_thm} below shows a global minimizer is readily computable regardless of the constraint set. Moreover, 
the solution to \eqref{eqn:main_linear_model} comes down to performing a single spectral decomposition and solving a single linear program. 

\begin{theorem}\label{thm:gen_lin_obj_thm}
Let $\bm{C} \in \R^{n\times n}$ with $\frac{1}{2}(\bm{C} + \bm{C}^\top) = \bm{P} \bm{\Omega} \bm{P}^\top$ for $\bm{P} \in \mathcal{O}(n)$ and $\bm{\Omega}=\Diag([\omega_{1}, \omega_{2}, \hdots, \omega_{n}])$ with $\omega_1 \geq \omega_2 \geq \hdots \geq \omega_n$. Then a global minimizer of \eqref{eqn:main_linear_model} is given by 
\begin{equation}
\bX^* = \bm{P}\Diag([\lambda^*_n, \lambda^*_{n-1}, \hdots, \lambda^*_1])\bm{P}^\top
\end{equation}
where 
\begin{equation}\label{eqn:lambda_opt_linear_model}
\blambda^* \in \argmin\left\{  \sum_{i=1}^{n} \omega_{i} \lambda_{n+1-i} \; | \;  \bA \blambda \leq {\bb}, \; \bm{D}_n \blambda \leq \bm{0} \right\}.
\end{equation}
\end{theorem}

\begin{proof}
We may assume without loss of generality $\bm{C} \in \Symn$. Observe, 
\[
\langle \bC, \bX \rangle = \bigg\langle \frac{1}{2}(\bC + \bC^\top), \bX \bigg\rangle
\]
for all $\bC \in \R^{n\times n}$, so \eqref{eqn:main_linear_model} can be rewritten such that the objective function is the inner product of two symmetric matrices. Thus, for the remainder of the proof, we assume $\bC\in \Symn$ and $\bC = \bP \bm{\Omega} \bP^\top$ for orthogonal $\bP$ and $\bm{\Omega} = \Diag([\omega_{1}, \omega_{2}, \hdots, \omega_{n}])$ with $\omega_1 \geq \omega_2 \geq \hdots \geq \omega_n$. Utilizing the spectral decomposition of $\bX$,
\begin{align}\label{eqn:lin_model_rankn}
\min &\left\{ \langle \bC, \bX \rangle \;| \; {\bA} \blambda(\bX) \leq {\bb}, \; \bX \in \Symn \right\} \nonumber \\
&\hspace{0.25in}=  \min \left\{ \langle \bP \bm{\Omega} \bP^\top, \bQ \;\text{Diag}(\blambda) \bQ^\top \rangle \;| \; \bA \blambda \leq {\bb}, \; \bm{D}_n \blambda \leq \bm{0}, \;\bQ \in \mathcal{O}(n) \right\}   \nonumber \\
&\hspace{0.25in}= \min \left\{ \langle \bm{\Omega}, \bP^\top\bQ\; \text{Diag}(\blambda) \bQ^\top \bP \rangle \;| \; \bA \blambda \leq {\bb}, \; \bm{D}_n \blambda \leq \bm{0}, \;\bQ \in \mathcal{O}(n) \right\}  \nonumber \\
&\hspace{0.25in}= \min \left\{ \langle \bm{\Omega}, \bar{\bQ}^\top\; \text{Diag}(\blambda) \bar{\bQ}\rangle \;| \; \bA \blambda \leq {\bb}, \; \bm{D}_n \blambda \leq \bm{0}, \;\bar{\bQ} \in \mathcal{O}(n) \right\}  \nonumber \\
&\hspace{0.25in}= \min \left\{ \text{Tr}\left( \bm{\Omega}\bar{\bQ}^\top\; \text{Diag}(\blambda) \bar{\bQ}\right) \;| \; \bA \blambda \leq {\bb}, \; \bm{D}_n \blambda \leq \bm{0}, \;\bar{\bQ} \in \mathcal{O}(n) \right\}.
\end{align}
By Theorem 2.1 in \cite{liang2023generalizing}, it follows for any $\blambda \in \R^n$ satisfying $\bA \blambda \leq {\bb}, \; \bm{D}_n \blambda \leq \bm{0}$ that  
\begin{equation}\label{eqn:LP_main}
\min \left\{ \text{Tr}\left( \bm{\Omega}\bar{\bQ}^\top\; \text{Diag}(\blambda) \bar{\bQ}\right) \;| \;\bar{\bQ} \in \mathcal{O}(n) \right\} = \sum_{i=1}^{n} \omega_{i} \lambda_{n+1-i}. \nonumber 
\end{equation}
Therefore, 
\[
\min \left\{ \langle \bC, \bX \rangle \;| \; {\bA} \blambda(\bX) \leq {\bb}, \; \bX \in \Symn \right\} = \min\left\{\sum_{i=1}^{n} \omega_i \lambda_{n+1-i} \; | \; \bA \blambda \leq {\bb}, \; \bm{D}_n \blambda \leq \bm{0} \right\}.
\]
Let 
$\blambda^* \in \argmin\left\{  \sum_{i=1}^{n} \omega_{i} \lambda_{n+1-i} \; | \; \bA \blambda \leq {\bb}, \; \bm{D}_n \blambda \leq \bm{0}\right\}.
$
In-order to compute $\bX^*$ such that the constraints are satisfied and 
$
\Tr(\bC^\top \bX^*) = \sum_{i=1}^{n} \omega_i \lambda_{n+1-i}^*,
$
we must determine $\bar{\bQ}^*$ such that, 
\begin{equation}\label{eqn:Qstar}
\text{Tr}\left(\bm{\Omega}(\bar{\bQ}^*)^\top \text{Diag}(\blambda^*)\bar{\bQ}^*\right) = \sum_{i=1}^{n} \omega_i \lambda_{n+1-i}^*.
\end{equation}
With $\bar{\bQ}^*$ defined by \eqref{eqn:Qstar} and using \eqref{eqn:lin_model_rankn} to see $\bar{\bQ} = \bQ^\top \bP$, 
a global minimizer $\bX^*$ is given by, 
\[
\bX^* = \bm{P} (\bar{\bQ}^*)^\top \text{Diag}(\blambda^*) \bar{\bQ}^* \bm{P}^\top.
\]
Since $\bar{\bQ}^* = [\bm{e}_n \; \bm{e}_{n-1} \hdots \bm{e}_1]$ solves \eqref{eqn:Qstar}, we obtain our final result.
\end{proof}

\begin{remark}
Though Theorem \ref{thm:gen_lin_obj_thm} was stated and proven for $\bC\in \R^{n\times n}$ and $\bX \in \Symn$, the same general argument applies with slight modifications if $\bX \in \mathbb{C}^{n\times n}$ is Hermitian and $\bC \in \mathbb{C}^{n \times n}$.   
\end{remark}

Theorem \ref{thm:gen_lin_obj_thm} provides a straight-forward procedure for computing a global minimizer to \eqref{eqn:main_linear_model} by solving a single linear program and performing one spectral decomposition, and the result was independent of the convexity of $\mathcal{S}_{\blambda}(\bA,\bb)$; therefore, \eqref{eqn:main_linear_model} is solvable in polynomial-time and constitutes a reasonable subproblem for an algorithm. 

\section{Projecting onto the Eigenvalue Constraint Set}\label{sec:projection}
Projecting onto the constraint set is a crucial procedure in numerous constrained optimization algorithms. In this section, we utilize an argument similar to the one applied in Section \ref{sec:lin_obj} to compute projections onto $\mathcal{S}_{\blambda}(\bA,\bb)$. That is, given $\bm{Y} \in \R^{n\times n}$, we solve 
\begin{align}\label{eqn:projection_model}
\min&\; \frac{1}{2}\|\bX - \bY\|_F^2 \\
\text{s.t.}&\; \bA \lambda(\bX) \leq \bb \nonumber \\
           &\; {\color{black} \bX \in \Symn}. \nonumber 
\end{align}

\begin{theorem}\label{thm:projection_sol}
Let $\bY \in \R^{n\times n}$ with $\frac{1}{2}(\bY+\bY^\top) = \bm{P} \bm{\Omega} \bm{P}^\top$ for $\bm{P}\in \mathcal{O}(n)$ and $\bm{\Omega} = \Diag([\omega_1, \omega_2, \hdots, \omega_n])$ with $\omega_1\leq \omega_2 \leq \hdots \leq \omega_n$ and $\bar{\bm{\omega}}:=[\omega_n,\omega_{n-1},\hdots, \omega_1]^\top$. Then 
\[
\bX^* = \bm{P}\Diag([\lambda^*_n, \lambda^*_{n-1}, \hdots, \lambda^*_1])\bm{P}^\top  
\]
is a global minimizer of \eqref{eqn:projection_model} where
\begin{equation}\label{eqn:lambdastar_proj}
\blambda^* \in \argmin\left\{ \frac{1}{2} \| \blambda - \bar{\bm{\omega}} \|^2_2 \; | \;{\bA} \blambda \leq {\bb},\; \bm{D}_n \blambda \leq \bm{0} \right\}.
\end{equation}
\end{theorem}

\begin{proof}
Let $\frac{1}{2}(\bY+\bY^\top) = \bm{P} \bm{\Omega} \bm{P}^\top$ with $\bm{P}\in \mathcal{O}(n)$ and $\bm{\Omega} = \Diag([\omega_1, \omega_2, \hdots, \omega_n])$ such that $\omega_1\leq \omega_2 \leq \hdots \leq \omega_n$. Define $\bar{\bm{\Omega}}:= -\bm{\Omega}$ and $\bar{\bm{\omega}}:=[\omega_n,\omega_{n-1}, \hdots, \omega_1]^\top$. Utilizing the spectral decomposition of $\bX$, 
\begin{align}
\min&\left\{ \frac{1}{2} \| \bX - \bY \|_F^2 \; | \; \bA \blambda(\bX) \leq \bb, \bX \in \Symn \right\} \nonumber \\
&= \min\left\{ \frac{1}{2} \| \bX \|_F^2 - \langle \bX, \bY \rangle \; | \; \bA \blambda(\bX) \leq \bb, \bX \in \Symn \right\} + \frac{1}{2}\|\bY\|_F^2 \nonumber \\
&=\min\left\{ \frac{1}{2} \| \blambda \|^2 - \langle \bQ \Diag(\blambda) \bQ^\top, \bm{P}\bm{\Omega}\bm{P}^\top \rangle \; | \; {\bA} \blambda \leq {\bb},\; \bm{D}_n \blambda \leq \bm{0}, \bQ \in \mathcal{O}(n) \right\} + \frac{1}{2}\|\bY\|_F^2 \nonumber \\
&=\min\left\{ \frac{1}{2} \| \blambda \|^2 + \langle \bar{\bQ}^\top \Diag(\blambda) \bar{\bQ}, \bar{\bm{\Omega}} \rangle \; | \; {\bA} \blambda \leq {\bb},\; \bm{D}_n \blambda \leq \bm{0}, \bar{\bQ} \in \mathcal{O}(n) \right\} + \frac{1}{2}\|\bY\|_F^2 \nonumber \\
&=\min\left\{ \frac{1}{2} \| \blambda \|^2 - \langle \bar{\bm{\omega}}, \blambda\rangle \; | \; {\bA} \blambda \leq {\bb},\; \bm{D}_n \blambda \leq \bm{0}\right\} + \frac{1}{2}\|\bY\|_F^2 \nonumber \\
&=\min\left\{ \frac{1}{2} \| \blambda - \bar{\bm{\omega}} \|^2_2 \; | \;{\bA} \blambda \leq {\bb},\; \bm{D}_n \blambda \leq \bm{0} \right\} + \frac{1}{2}(\|\bY\|_F^2 - \|\bar{\bm{\omega}}\|^2_2 )  \nonumber \\
&=\;\;\frac{1}{2} \left(  \| \text{Proj}_{\mathcal{C}}(\bar{\bm{\omega}}) - \bar{\bm{\omega}} \|^2_2 + \|\bY\|_F^2 - \|\bar{\bm{\omega}}\|^2_2 \right) 
\end{align}
where $\mathcal{C}:=\left\{\bx \in \R^n \; | \; {\bA} \bx \leq {\bb},\; \bm{D}_n \bx \leq \bm{0}   \right\}$ and $\text{Proj}_{\mathcal{C}}(\cdot)$ is the projection operator onto $\mathcal{C}$. Note, the fourth equality above follows from Theorem 2.1 in \cite{liang2023generalizing}. Let $\blambda^*:=\text{Proj}_{\mathcal{C}}(\bar{\bm{\omega}})$. Note, if $\bar{\bQ}^* = [\bm{e}_n\; \bm{e}_{n-1}\; \hdots\; \bm{e}_1]$, then 
\begin{equation}\label{eqn:Qstar_proj}
\Tr(\bar{\bm{\Omega}}(\bar{\bQ}^*)^\top \Diag(\blambda^*) \bar{\bQ}^*)= \frac{1}{2} \left( \|\blambda^*-\bar{\bm{\omega}}\|^2 - \|\blambda^*\|^2 - \|\bar{\bm{\omega}}\|^2 \right) = \langle \blambda^*, -\bar{\bm{\omega}}\rangle\nonumber 
\end{equation}
and by the change-of-variable above $\bX^* = \bm{P} (\bar{\bQ}^*)^\top \Diag(\blambda^*) \bar{\bQ}^* \bm{P}^\top$ is a global minimizer of \eqref{eqn:projection_model}.
\end{proof}


Thus, similar to the linear objective problem, projecting onto the eigenvalue constraint only requires computing a spectral decomposition and solving a convex optimization problem. In this case, a projection onto $\{\bx \in \R^n \; | \; \bA \bx \leq \bb, \; \bm{D}_n \bx \leq \bm{0}\}$ must be computed. This problem can be solved via any convex optimization solver, but specialized fast algorithms for projecting onto a polyhedron have been developed \cite{hager2016projection}.
We now utilize our results to develop two optimization algorithms which obtain first-order stationary points to \eqref{eqn:SCO-Eig} when the constraint set is convex. 
\section{Spectrally Constrained Solvers}\label{sec:solvers}
We develop two first-order algorithms to compute stationary points for \eqref{eqn:SCO-Eig}. Our capacity to solve the linear objective problem motivates the construction of a Frank-Wolfe algorithm and accessible projections make possible the development of a projected gradient method. 
\subsection{Inexact Projected Gradient Method}\label{sec:projected_gradient_algorithm}
For $\bY \in \R^{n \times n}$ let $\text{Proj}_{\mathcal{S}_{\lambda}}(\cdot)$ be the operator which maps $\bY$ to an element of the argmin set of \eqref{eqn:projection_model}, i.e., 
\[
\text{Proj}_{\mathcal{S}_{\lambda}}(\bY) \in \argmin \left\{\frac{1}{2}\| \bX - \bY\|_F^2 \; | \; \bX \in \mathcal{S}_{\lambda}(\bA,\bb) \right\}.
\]
For $\mathcal{S}_{\blambda}(\bA,\bb)$ convex, the projection is unique. If the set is non-convex multiple optimal projections are likely present. 
Since inexact computations are the reality in implementation, we introduce a notion of inexact projections.
\begin{definition}
For convex $\mathcal{S}_{\blambda}(\bA,\bb)$ and parameter $\delta\geq 0$, the set of inexact projections of $\bm{Y}\in \R^{n\times n}$ onto $\mathcal{S}_{\blambda}(\bA,\bb)$ is 
\begin{equation}\label{eqn:inexact_proj}
\text{P}_{\cS_{\blambda}}^{\delta}(\bm{Y}):= \left\{ \bm{Z} \in \cS_{\blambda}(\bA,\bb) \; | \; \langle \bm{Z} - \bm{Y}, \bX - \bm{Z} \rangle \geq -\delta, \; \forall \bX \in \cS_{\blambda}(\bA,\bb) \right\}. 
\end{equation}
\end{definition}
The standard results for projections onto convex sets tell us $\bZ^*$ is the optimal projection of $\bY$ onto convex $\mathcal{S}_{\blambda}(\bA,\bb)$ if and only if 
\[
\langle \bm{Z}^* - \bm{Y}, \bX - \bm{Z}^* \rangle \geq 0, \; \forall \bX \in \cS_{\blambda}(\bA,\bb).
\]
Therefore, the set of inexact projections are the points which inexactly satisfy this condition where the level of inexactness is controlled by $\delta$. 
Clearly, $\text{P}_{\cS_{\blambda}}^{0}(\bm{Y}) = \text{Proj}_{\cS_{\blambda}}(\bm{Y})$ when $\cS_{\blambda}(\bA,\bb)$ is a convex set. 

For the sake of our analysis, we make the following assumption throughout Section \ref{sec:solvers}
\begin{assumption}\label{assumption_1}
 $\mathcal{S}_{\blambda}(\bA,\bb)$ is convex.
\end{assumption}
This assumption enables us to have a clear notion of first-order $\epsilon$-stationary points for \eqref{eqn:SCO-Eig}.
\begin{definition}\label{def:stat_pt}
Let $\epsilon \geq 0$. The point $\bX^* \in \mathcal{S}_{\blambda}(\bA,\bb)$ is a first-order $\epsilon$-stationary point of \eqref{eqn:SCO-Eig} provided
\[
\min \left\{ \langle \nabla F(\bX^*), \bX - \bX^* \rangle \; | \; \|\bX - \bX^*\|_F \leq 1, \; \bX \in \mathcal{S}_{\blambda}(\bA,\bb) \right\} \geq -\epsilon.
\]
\end{definition}
Algorithm \ref{alg:projected_grad} presents our inexact projected gradient method for solving \eqref{eqn:SCO-Eig}. 
The method can be neatly described. A point $\bX$ is selected from $\cS_{\blambda}(\bA,\bb)$ and the traditional gradient update step is taken to compute $\bX_+ = \bX - \alpha \nabla F(\bX)$ where $\alpha >0$ is the stepsize. Since it is not necessarily true $\bX_+ \in \cS_{\blambda}(\bA,\bb)$, because $\alpha$ could be too long of a step or $\nabla F(\bX)$ might point out of the constraint set, feasibility is maintained by projecting $\bX_+$ onto the constraint. To save computational expenses, the projection is done inexactly. Line-search is performed at each iteration to ensure a sufficient decrease is obtained. The algorithm terminates when the gap between consecutive iterates decreases below a provided tolerance. 

The approach is a fairly straightforward implementation of the traditional method in constrained optimization. The key novelty of Algorithm \ref{alg:projected_grad} comes from the fact we can solve the projection problem by Theorem \ref{thm:projection_sol}. Without this result the algorithm would not be implementable.
%
%
%
%
%

%
\begin{algorithm}
\caption{Inexact Projected Gradient Method}\label{alg:projected_grad}
\begin{algorithmic}[1]
    \Require $\bX_0 \in \mathcal{S}_{\blambda}(\bA,\bb)$; \; $\epsilon > 0$; \; $\delta \in [0,1)$; \; $\alpha > 0$;\; $\tau_1 \in (0,1)$;\; $h > 0$  
\For{$k=0,1,2 \hdots$} 
\State $h_k = h$ 
\State $\bX_{k+1} = \text{P}_{\mathcal{S}_{\lambda}}^{\delta}(\bX_k - h_k \nabla F(\bX_k))$ 
\If{$\|\bX_k - \bX_{k+1}\|_F \leq  \epsilon$}
    \State \textbf{Return}\; $\bX_k$ 
\Else
\While{$F(\bX_{k+1}) > F(\bX_k) - \alpha \| \bX_{k+1} - \bX_k\|_F^2$} 
    \State $h_k = \tau_1 \cdot h_k$
    \State $\bX_{k+1} = \text{P}_{\mathcal{S}_{\lambda}}^{\delta}(\bX_k - h_k \nabla F(\bX_k))$
\EndWhile
\EndIf
\EndFor
\end{algorithmic}
\end{algorithm}

We now explicitly state the convergence result of Algorithm \ref{alg:projected_grad}, where it is proved in the Appendix.
\begin{theorem}\label{thm:projected_grad_Thm}
Assume $F$ is gradient Lipschitz with parameter $L>0$, the initial level set, i.e., $\left\{ \bX \in \cS_{\blambda}(\bA,\bb) \;|\; F(\bX) \leq F(\bX_0) \right\}$, is a bounded subset of $\Symn$ with diameter $D$, there exists $F^*$ such that $F^* \leq F(\bX)$ for all $\bX\in \mathcal{S}_{\blambda}(\bA,\bb)$, there exists $M>0$ such that $\|\nabla F(\bX_k)\|_F \leq M$ for all $k\geq 0$, and Assumption 1 holds. 
If inexact projections are computed with sufficient accuracy, dependent on $\epsilon$, then Algorithm \ref{alg:projected_grad} will converge to a first-order $\epsilon$-stationary point of \eqref{eqn:SCO-Eig} in $\mathcal{O}(\epsilon^{-2})$ iterations. 
More explicitly, a first-order $\epsilon$-stationary point will be returned after no more than 
\[
  \left( \frac{4(D+Mh_{\text{low}}+1)^2}{h_{\text{low}}^2}\right) \left(\frac{F(\bX_0) - F^*}{\alpha}\right)\frac{1}{\epsilon^2}                                    \text{ iterations} 
\]
provided the accuracy of the inexact projections, $\delta$, satisfies $\delta \leq \min\left\{ \frac{h_{\text{low}}}{2}\epsilon, \; \frac{1}{2}\epsilon_{\text{tol}}^2\right\}$, where \\
$h_{\text{low}} := \tau_1/(L+2\alpha)$ 
and 
$
\epsilon_{\text{tol}}:= \frac{h_{low} \epsilon}{2(D+Mh_{low}+1)}. 
$
 
\end{theorem}

\subsection{Inexact Frank-Wolfe Algorithm}\label{sec:frankwolfe}
We now present a Frank-Wolfe algorithm for solving \eqref{eqn:SCO-Eig}. 
The algorithm and analysis we present are an extension of the work in the concise technical report of Lacoste-Julien \cite{lacoste2016convergence}. 
The author in this paper presents the first Frank-Wolfe method for smooth non-convex functions over a compact and convex constraint.
Our approach extends the work in \cite{lacoste2016convergence} by removing the compactness assumption. To accomplish this, we utilized a different notion of first-order stationarity, replaced the compactness assumption with the weaker assumption that the initial level set of the objective function is bounded, and introduced a modified subproblem for our model.  
Traditional Frank-Wolfe approaches would require solving subproblems of the form
\begin{align}\label{eqn:SCO_eig_linear+ball}
\min&\; \langle \bC, \bX - \bX_0 \rangle \\
\text{s.t.}&\; \bA \blambda(\bX) \leq \bb \nonumber \\
           &\; \|\bX-\bX_0\|_F \leq 1 \nonumber \\
           &\; \bX \in \Symn. \nonumber 
\end{align}
This model appears outside the scope of \eqref{eqn:SCO-Eig};
however, Theorem \ref{thm:FW_alg_alt_opts} below shows \eqref{eqn:SCO_eig_linear+ball} can be bounded by an alternative optimization model which only contains linear constraints on the eigenvalues. 
\begin{theorem}\label{thm:FW_alg_alt_opts}
Given any $\bX_0 \in \Symn$, we have the following upper bound to \eqref{eqn:SCO_eig_linear+ball}
\begin{multline}\label{eqn:opt_approx_lower}
\bigg| \min\left\{ \langle \bC, \bX - \bX_0 \rangle \; | \; \bX \in \mathcal{S}_{\blambda}(\bA,\bb), \;  \|\bX-\bX_0\|_{F} \leq 1 \right\}\bigg| \\ \leq 
\bigg|\min\left\{ \langle \bC, \bX - \bX_0 \rangle \; | \; \bX \in \mathcal{S}_{\blambda}(\bA,\bb), \;  \|\blambda(\bX)-\blambda(\bX_0)\|_{\infty} \leq 1 \right\}\bigg|.
\end{multline}
\end{theorem}

\begin{proof} From Section 6.3 of \cite{horn_johnson_1991}, we know for any $\bX, \bX_0 \in \Symn$
\[
\| \blambda(\bX) - \blambda(\bX_0)\|_2 \leq \|\bX - \bX_0\|_2 
\text{ which implies }  
\| \blambda(\bX) - \blambda(\bX_0)\|_{\infty} \leq \|\bX - \bX_0\|_F.  
\]
Thus, given any $\bX \in \Symn$ such that $\|\bX - \bX_0\|_F \leq 1$ it follows $\| \blambda(\bX) - \blambda(\bX_0)\|_{\infty} \leq 1$ and $  \left\{\bX \in \Symn \; | \;  \|\bX - \bX_0\|_F \leq 1\right\} \subset  \left\{\bX \in \Symn \; | \;  \| \blambda(\bX) - \blambda(\bX_0)\|_{\infty} \leq 1\right\}$. Hence 
\begin{multline}\label{eqn:opt_approxI}
\min\left\{ \langle \bC, \bX - \bX_0 \rangle \; | \; \bX \in \mathcal{S}_{\blambda}(\bA,\bb), \;  \|\bX-\bX_0\|_{F} \leq 1 \right\} \\ \geq  \min\left\{ \langle \bC, \bX - \bX_0 \rangle \; | \; \bX \in \mathcal{S}_{\blambda}(\bA,\bb), \;  \|\blambda(\bX)-\blambda(\bX_0)\|_{\infty} \leq 1 \right\}, 
\end{multline}
and we know the optimal value of both models is non-positive since $\bX_0$ is a feasible solution for each model. Therefore, the right-hand side of \eqref{eqn:opt_approxI} is greater in terms of absolute magnitude. 
\end{proof}
Theorem \ref{thm:FW_alg_alt_opts} results in our ability to approximate the solution to \eqref{eqn:SCO_eig_linear+ball} by solving an alternative optimization model which only appends additional linear inequality constraints on the eigenvalues. 
Since Definition \ref{def:stat_pt} is equivalent to the left-hand side of \eqref{eqn:opt_approx_lower} being less than $\epsilon$, it follows the right-hand side of \eqref{eqn:opt_approx_lower} provides an estimate of the optimality of $\bX_0$ via a tractable subproblem. Therefore, in our Frank-Wolfe approach, instead of solving \eqref{eqn:SCO_eig_linear+ball}, we solve models of the form 
\[
\min\left\{ \langle \bC, \bX - \bX_0 \rangle \; | \; \bX \in \mathcal{S}_{\blambda}(\bA,\bb), \;  \|\blambda(\bX)-\blambda(\bX_0)\|_{\infty} \leq 1 \right\}
\]
and use these solutions to bound the first-order stationarity condition of \eqref{eqn:SCO-Eig}. 

Algorithm \ref{alg:frank_wolfe} is our proposed inexact Frank-Wolfe algorithm for \eqref{eqn:SCO-Eig}. At each iteration it seeks to improve upon the current iterate by minimizing a first-order approximation of the function over the constraint set. We approximate the traditional Frank-Wolfe subproblem by Theorem \ref{thm:FW_alg_alt_opts}. The subproblem then produces a direction which will decrease the value of the objective function, and a stepsize is determined which maintains the feasibility of the next iterate. 

\begin{algorithm}
\caption{Inexact Frank-Wolfe algorithm}\label{alg:frank_wolfe}
\begin{algorithmic}[1]
\Require $\bm{X}_0 \in \mathcal{S}_{\blambda}(\bA,\bb)$; \; $\epsilon > 0$;\;  $\alpha \in [0,1)$;\; $\delta \in [0,\alpha \epsilon)$; \; $\Theta > 0$
\For {$k=0,1,2, \hdots$}
    \State Approximately compute a solution, $\bm{D}_k$, to  
    \vspace{-0.1in}
\begin{equation}\label{eqn:FW_subprob}    
m_k^*:=\min\left\{ \langle \nabla F(\bX_k), \bm{D} - \bX_k \rangle \; | \; \bm{D} \in \mathcal{S}_{\blambda}(\bA,\bb), \;  \|\blambda(\bm{D})-\blambda(\bX_k)\|_{\infty} \leq 1 \right\} 
\end{equation}
    \State such that $m_k:= \langle \nabla F(\bX_k), \bm{D}_k - \bX_k\rangle \leq 0$ and is within $\delta$ of $m_k^*$
    \If{$|m_k| \leq (1-\alpha)\epsilon$}
        \State \textbf{Return}\; $\bX_k$ 
    \Else
        \State Set $\gamma_k := \min \left\{ |m_k|/\Theta, 1 \right\}$ 
        \State Update $\bX_{k+1}:=\bX_k + \gamma_k \left( \bm{D}_k - \bX_k\right)$
    \EndIf
\EndFor
\end{algorithmic}
\end{algorithm}

Theorem \ref{thm:FW_alg} states the sublinear convergence rate of Algorithm \ref{alg:frank_wolfe} to stationary points of \eqref{eqn:SCO-Eig}.

\begin{theorem}\label{thm:FW_alg}
Assume  $F$ is gradient Lipschitz with parameter $L>0$, the initial level set is bounded, there exists $F^*$ such that $F^* \leq F(\bX)$ for all $\bX\in \mathcal{S}_{\blambda}(\bA,\bb)$, and Assumption 1 holds. Define 
\[
\rho:=  \sup \left\{\|\bY - \bX\|_F^2 \; | \; F(\bX) \leq F(\bX_0), \; \| \blambda(\bX) - \blambda(\bY)\|_{\infty} \leq 1, \; \bY, \bX \in \mathcal{S}_{\blambda}(\bA,\bb) \right\}.
\]
If $\Theta \geq \rho L$, then Algorithm \ref{alg:frank_wolfe} will compute a first-order $\epsilon$-stationary point to \eqref{eqn:SCO-Eig} in $\mathcal{O}(\epsilon^{-2})$ iterations. More precisely, a first-order $\epsilon$-stationary point will be obtained after no more than 
\[
K \geq \bigg\lceil \frac{\max\left\{2 (F(\bX_0) - F^*), \Theta \right\}^2}{(1-\alpha)^2 \epsilon^2}-1 \bigg\rceil   \text{ iterations.}
\]

\end{theorem}

\begin{proof}
We first show $\rho$ provides a usable bound for all iterations of Algorithm \ref{alg:frank_wolfe}. If $|m_0| \leq \Theta$, then by the gradient Lipschitz condition on $F$ we have,  
\begin{align}\label{eqn:F_bound_1}
F(\bX_{1}) &\leq F(\bX_0) + \gamma_0\langle \nabla F(\bX_0), \bm{D}_0 - \bX_0 \rangle + \frac{L\gamma_0^2}{2}\| \bm{D}_0 - \bX_0\|_F^2 \nonumber \\
&\leq F(\bX_0) + \gamma_0\langle \nabla F(\bX_0), \bm{D}_0 - \bX_0 \rangle + \frac{\rho L\gamma_0^2}{2} \nonumber \\
&\leq F(\bX_0) + \gamma_0 m_0 + \frac{\Theta}{2}\gamma_0^2 \nonumber\\
&\leq F(\bX_0) -\frac{|m_0|^2}{\Theta} + \frac{1}{2\Theta}|m_0|^2 \nonumber \\
&= F(\bX_0) - \left(\frac{1}{2\Theta}\right)|m_0|^2
\end{align}
where the second inequality follows from the definition of $\rho$. If instead $|m_0| > \Theta$, then we have 
\begin{equation}\label{eqn:F_bound_2}
F(\bX_{1}) \leq F(\bX_0) - |m_0| + \frac{\Theta}{2} < F(\bX_0) - \frac{\Theta}{2}.
\end{equation}
Therefore, regardless of how the stepsize is computed, $F(\bX_1) \leq F(\bX_0)$ which implies $\bX_1$ is contained in the initial level set. So, the bound provided by $\rho$ remains valid for all iterations and by induction on the iteration count we have   
\begin{equation}\label{eqn:FW_eq_help}
F(\bX_{k+1}) \leq F(\bX_k) - \min\left\{\frac{|m_k|^2}{2\Theta}, |m_k| - \frac{\Theta}{2}\mathbf{1}_{\left\{|m_k| > \Theta\right\}} \right\}. 
\end{equation}
Thus, using the same argument presented for Theorem 1 in \cite{lacoste2016convergence} it follows for all $K \geq 0$,
\[
\min_{0\leq k \leq K} |m_k| \leq \frac{\max\left\{2(F(\bX_0) - F^*), \Theta \right\}}{\sqrt{K+1}}.
\]
Since $m_k$ measures the first-order stationary condition inexactly, it then follows an $\epsilon$-stationary point will be obtained provided $|m_t| \leq (1-\alpha)\epsilon$ where $t \in \{0,1\hdots, K\}$ produces the minimum value of $|m_k|$ over all the iterates since $|m_t^*| \leq |m_t| \leq |m_t| + \delta  \leq \epsilon$. Finally, by \eqref{eqn:opt_approx_lower} it follows if $|m_t^*| \leq \epsilon$ then  
\[
\bigg| \min\left\{ \langle \bC, \bX - \bX_t \rangle \; | \; \bX \in \mathcal{S}_{\blambda}(\bA,\bb), \;  \|\bX-\bX_t\|_{F} \leq 1 \right\}\bigg|  \leq \epsilon. 
\]
\vspace{-0.2in}
\end{proof}
The condition in Theorem \ref{thm:FW_alg} requiring $\Theta \geq \rho L$ might appear unruly since both $\rho$ and $L$ could be unknown. In practice this difficulty is overcome by updating $\Theta$ on an iterative basis. For example, if $\Theta \geq \rho L$, then the inequality in \eqref{eqn:FW_eq_help} shall hold. If this inequality fails, then $\Theta$ can be increased until the inequality is satisfied at the current iterate. In our implementation of Algorithm \ref{alg:frank_wolfe}, we update the value of $\Theta$ at each iteration depending on whether or not the value of the objective function was improved. This is why we ensure an improvement of the value of the objective function. 

Both Algorithms \ref{alg:projected_grad} and \ref{alg:frank_wolfe} converge sublinearly to first-order stationary points of \eqref{eqn:SCO-Eig} under reasonable assumptions. The convexity assumption is necessary for the current convergence analysis of both algorithms and is required to maintain the feasibility of the iterates generated by Algorithm \ref{alg:frank_wolfe}. The projected gradient method however can be applied successfully to instances of \eqref{eqn:SCO-Eig} with non-convex constraints and maintain feasibility. We observe this in Section \ref{sec:quadratic_equations} when Algorithm \ref{alg:projected_grad} is applied to solve systems of quadratic equations.

\begin{remark}\label{rem:block_model}
The theory and algorithms we have developed extend to the block optimization model
\begin{align}\label{eqn:block_SCO}
\min&\; F(\bX_1, \hdots, \bX_k) \\
\text{s.t.}&\; \bX_i \in \mathcal{S}_{\blambda}(\bA_i, \bb_i), \; i = 1, \hdots, k. \nonumber  
\end{align}
Letting $\mathcal{V}:=\mathcal{S}^{n_1 \times n_1} \times \cdots \times \mathcal{S}^{n_k \times n_k}$ with vectors $\mathcal{X} = (\bX_1, \hdots, \bX_k)$ and associated inner product $\langle \cX, \cY \rangle_{\cV}:= \sum_{i=1}^{k} \langle \bX_i, \bY_i\rangle$,  
the analysis in the Appendix proves Algorithm \ref{alg:projected_grad} computes first-order stationary points of \eqref{eqn:block_SCO} where exact projections onto $\mathcal{S}_{\blambda}(\bA_1, \bb_1) \times \cdots \times \mathcal{S}_{\blambda}(\bA_k, \bb_k)$ are readily computed by projecting onto each component $\mathcal{S}_{\blambda}(\bA_i, \bb_i)$. Similarly, the required subproblem to implement a modified version of Algorithm \ref{alg:frank_wolfe} on \eqref{eqn:block_SCO} is equivalent to solving $k$ subproblems of the form described in Section \ref{sec:frankwolfe}. Proving convergence of a modified version of Algorithm \ref{alg:frank_wolfe} to stationary points of \eqref{eqn:block_SCO} only requires minor alterations to the current analysis.
\end{remark}

\section{Numerical Experiments}\label{sec:experiments}
The \eqref{eqn:SCO-Eig} paradigm applies to many constrained matrix problems: semidefinite programming \cite{vandenberghe1996semidefinite}, condition number constraints \cite{tanaka2014positive,won2013condition}, and rank constrained optimization \cite{li2023partial,zhu2018global} to list a few. 
In this section, we first demonstrate the performance of Algorithms \ref{alg:projected_grad} and \ref{alg:frank_wolfe} on a popular preconditioning model to showcase how the convergence theory aligns with what is observed in practice. Then, we demonstrate the applicability of our methodology for solving systems of quadratic equations. Our numerical results demonstrate our method can outperform classical approaches for solving quadratic systems such as Newton's method. 

\subsection{Preconditioning}\label{sec:precond_example}
One could argue solving systems of linear equations is the most important task in applied mathematics. Many iterative methods have been devised to solve  $\bA \bx = \bb$ with $\bA \in \R^{m \times n}$ full rank and $\bb \in \R^m$ such as the Jacobi method \cite{saad2003iterative} and the conjugate gradient method \cite{hestenes1952methods}. Methods for solving linear systems often converge linearly with the rate dependent upon the condition number of the matrix $\bA$, i.e., 
\[
\kappa(\bA):= \frac{\sigma_{\max}(\bA)}{\sigma_{\min}(\bA)}
\]
where $\sigma_{\max}(\bA)$ and $\sigma_{\min}(\bA)$ are the largest and smallest singular values of $\bA$ respectively. If $\bA \in \mathcal{S}_{++}^{n \times n}$, then $\kappa(\bA) = \lambda_1(\bA)/\lambda_n(\bA)$. Table 1 in \cite{qu2022optimal} compares how the convergence rates of different iterative methods depend on $\kappa(\bA)$. The art of matrix preconditioning is to form an equivalent linear system with an improved condition number which can be solved faster by iterative methods. This is accomplished by multiplying $\bA$ by simple matrices $\bX$ and $\bY$ such that $\bA\bX$, $\bY \bA$ or $\bY \bA \bX$ have an improved condition number. Preconditioning is a well-studied aspect of numerical linear algebra \cite{chen_2005,wathen_2015} and numerous approaches have been proposed. For example, a recent paper proposes optimal diagonal preconditioners, i.e., $\bX$ and $\bY$ are diagonal matrices \cite{qu2022optimal}. One popular preconditioning model proposed by Benson \cite{benson1982iterative} is 
\begin{align}\label{eqn:benson_precond}
\min&\; \| \bA \bX - \bm{I} \|_F \\
\text{s.t.}&\; \bX \in \mathcal{M} \nonumber 
\end{align}
where $\mathcal{M}$ is a special set of matrices. To reduce computation, Benson \cite{benson1982iterative} assumed sufficiently sparse matrices in $\mathcal{M}$, e.g., positive diagonal matrices. We demonstrate the performance of our algorithms by solving \eqref{eqn:benson_precond} with various convex eigenvalue constraints, where we do not restrict to sparse matrices: 
\begin{equation}\label{eqn:S1}
\mathcal{M}_1 := \left\{ \bX \in \Symn \; | \; \lambda_i(\bX) \in [0.001,1], \; i=1,\hdots, n   \right\},
\end{equation}
\begin{equation}\label{eqn:S2}
\mathcal{M}_2 :=  \left\{ \bX \in \Symn \; | \;  \lambda_1(\bX) - \kappa \lambda_n(\bX) \leq 0, \; \lambda_n(\bX) \geq 0 \right\},
\end{equation}
and
\begin{equation}\label{eqn:S3}
\mathcal{M}_3 :=  \left\{ \bX \in \Symn \; | \; \bm{c}_i^\top \lambda(\bX) \leq 1, \; i=1,\hdots, m  \right\},
\end{equation}
where $\kappa > 0 $ and $\bm{c}_i = [i, i-1, \hdots, 1, 0, \hdots, 0]^\top$ for all $i$. The constraints $\cM_1$ and $\cM_2$ are reasonable choices for the constraint in \eqref{eqn:benson_precond} because they enforce the preconditioning matrix to be well-conditioned. The first ensures the condition number of the preconditioning matrix $\bX$ is bounded above by 1000 with bounded eigenvalues between 0.001 and 1; the second ensures the condition number is bounded above by $\kappa$ while not enforcing an upper bound on the eigenvalues. Note, $\cM_2$ does admit the zero matrix as feasible which has an undefined condition number, but this is often of no consequence because often non-zero matrices are present in $\cM_2$ which yield better objective function values. The purpose of the last constraint is to showcase the algorithms applied to a general convex set as generated by Theorem \ref{thm:convex_constraint}. The goal of these experiments is to evince the functionality of Algorithms \ref{alg:projected_grad} and \ref{alg:frank_wolfe} applied with different constraints on an objective function of practical import. Developing new specialized approaches for preconditioning is outside the scope of this paper.   

We applied Algorithms \ref{alg:projected_grad} and \ref{alg:frank_wolfe} on three different instances of \eqref{eqn:benson_precond} with $\cM_1$, $\cM_2$, and $\cM_3$ serving as the constraint space $\cM$. The matrix $\bA \in \mathcal{S}^{250 \times 250}$ to be preconditioned was generated as $\bA = \bm{V} \bm{V}^\top$ where each element of $\bm{V}\in \R^{250 \times 250}$ was drawn from a standard normal. Figure \ref{fig:Alg_plots} displays the convergence plots of the experiments.  

\begin{figure}[h]
\centering
\begin{subfigure}{0.48\textwidth}
    \includegraphics[width=\textwidth]{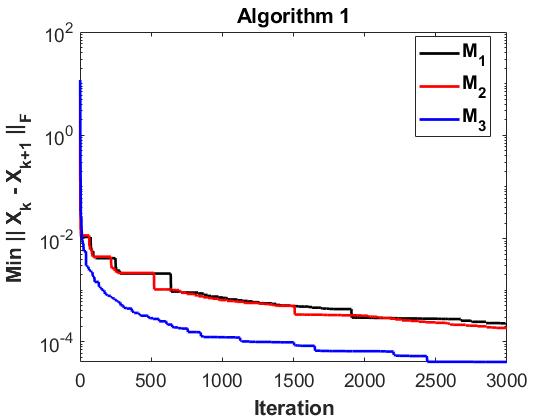}
    \label{fig:first}
\end{subfigure}
\begin{subfigure}{0.48\textwidth}
    \includegraphics[width=\textwidth]{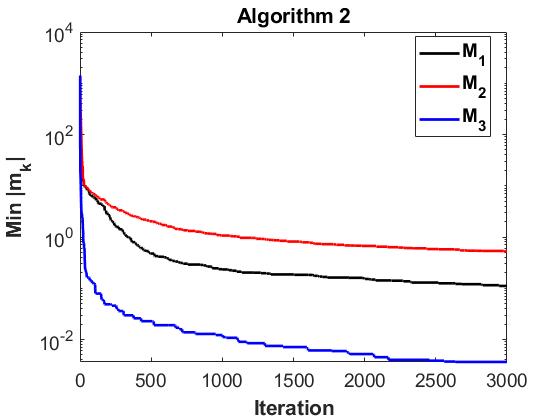}
    \label{fig:second}
\end{subfigure}       
\vspace{-0.2in}
\caption{Convergence plots of the first-order $\epsilon$-stationary condition for Algorithms \ref{alg:projected_grad} and \ref{alg:frank_wolfe} applied to \eqref{eqn:benson_precond}. The y-axis is in log-scale. For Algorithm \ref{alg:projected_grad} the y-axis measures the minimum distance between consecutive iterates at iteration $k$. For Algorithm \ref{alg:frank_wolfe} the y-axis measures the best solution obtained to \eqref{eqn:FW_subprob} at iteration $k$.}
\label{fig:Alg_plots}
\end{figure}
Both methods were initialized at the identity matrix and ran for a total of 3000 iterations. The convergence plots display clearly the sublinear convergence rates guaranteed by Theorems \ref{thm:projected_grad_Thm} and \ref{thm:FW_alg}. 
These numerical results demonstrate the algorithms' performance aligns with the developed theory. We now turn to an application where our methodology outperforms classical approaches. 

\subsection{Solving Systems of Quadratic Equations}\label{sec:quadratic_equations}
Solving systems of quadratic equations occurs regularly across scientific domains.  
Two highly studied examples are phase retrieval and the algebraic Riccati equations. Phase retrieval is a problem of recovering a signal from the magnitude of its Fourier transform. Phase retrieval problems are important in imagining science with applications in X-ray crystallography \cite{harrison1993phase,millane1990phase}, X-ray tomography \cite{dierolf2010ptychographic} and astronomy \cite{fienup1987phase}, and all phase retrieval problems can be formulated as a system of quadratic equations \cite{candes2015phase,candes2014solving,jaganathan2016phase,wang2017solving}. The algebraic Riccati equations are crucial in the study of stochastic and optimal control \cite{lancaster1995algebraic}, and they also are equivalently expressed as a system of quadratic equations. 

Here we demonstrate how to apply our methods to solve systems of quadratic equations. The general form of the problem we consider is:
\begin{equation}\label{eqn:quad_eq}
\textbf{Find}\;\; \bx \in \R^n \;\; \textbf{s.t.} \;\; \bx^\top \bm{Q}_i \bx  = b_i, \; i=1,\hdots m 
\end{equation}
where $\bm{Q}_i \in \R^{n\times n}$ and $\bb\in \R^m$. An equivalent matrix form of the problem is:
\begin{equation}
\textbf{Find}\;\; \bX \in \Symn_{+}, \; \text{rank}(\bX)=1 \;\; \textbf{s.t.} \;\; \langle \bm{Q}_i, \bX\rangle  = b_i, \; i=1,\hdots m  \nonumber 
\end{equation}
which naturally leads to the rank constrained model 
\begin{align}\label{eqn:quad_eq_matrix_opt}
\min&\; \sum_{i=1}^{m} \left( \langle \bm{Q}_i, \bX \rangle - b_i  \right)^2 \\
\text{s.t.}& \; \text{rank}(\bX)=1 \nonumber \\
           & \;  \bX \in \Symn_{+}. \nonumber 
\end{align}
If $\bX^*$ is computed yielding an objective value of zero to \eqref{eqn:quad_eq_matrix_opt}, then \eqref{eqn:quad_eq} has been solved. If no solution exists to \eqref{eqn:quad_eq}, then \eqref{eqn:quad_eq_matrix_opt} represents a least-squares approximate solution. A common relaxation of \eqref{eqn:quad_eq_matrix_opt} is to drop the rank constraint to obtain a convex model. A global solution to the convex relaxation can be computed and projected onto the set of rank-1 matrices to obtain an approximate solution to \eqref{eqn:quad_eq}. Another more precise relaxation of \eqref{eqn:quad_eq_matrix_opt} is made possible with our framework. We instead consider the model   
\begin{align}\label{eqn:quad_eq_matrix_opt_relax}
\min&\; \sum_{i=1}^{m} \left( \langle \bm{Q}_i, \bX \rangle - b_i  \right)^2 \\
\text{s.t.}& \; \blambda_i(\bX) \in [0, \delta], \;\; i=2,\hdots, n, \nonumber \\
           & \;  \bX \in \Symn \nonumber 
\end{align}
where $\delta >0$. For small $\delta$, the constraint closely approximates the rank-1 condition. It is easy to check this is a non-convex constraint on the eigenvalues which means the analysis in Section \ref{sec:solvers} does not guarantee convergence to stationary points; however, Theorem \ref{thm:projection_sol} guarantees we can compute optimal projections onto the constraint, so we can implement Algorithm \ref{alg:projected_grad} on \eqref{eqn:quad_eq_matrix_opt_relax}.

We considered three methods for solving \eqref{eqn:quad_eq}:
\begin{enumerate}
\item[]{\bf Newton:} Apply Newton's method directly to \eqref{eqn:quad_eq}. In our tests, we initialized and implemented Newton's method ten times and took the best result as the solution.
\item[]{\bf Convex + Newton:} Solve the convex relaxation of \eqref{eqn:quad_eq_matrix_opt}, project the solution onto the set of rank-1 matrices and then run Newton's method initialized from the projected solution. In our experiments, we utilized CVX to solve the convex relaxation \cite{cvx}.
\item[]{\bf SCO + Newton:} Apply Algorithm \ref{alg:projected_grad} to approximately solve \eqref{eqn:quad_eq_matrix_opt_relax}, project the solution onto the set of rank-1 matrices and then run Newton's method initialized from the projected solution. 
\end{enumerate}


Our experiments were performed on synthetic data. We simulated random instances of \eqref{eqn:quad_eq} by drawing each entry in the $\bm{Q}_i$'s from a standard normal distribution and projecting the resulting matrix onto the set of symmetric matrices. The coefficients $b_i$ were determined by randomly selecting a vector $\by \in R^n$ with entries from a standard normal and computing $\by^\top \bQ_i \by$ for all $i$. In this way, every system admitted a solution. We conducted ten numerical experiments; five with $(n,m)=(75,75)$ and five with $(n,m)=(100,100)$. The maximum allowed iterations for every instance of Newton's method was 5000, and the maximum iterations for Algorithm \ref{alg:projected_grad} was 10,000. We used $\delta = 1e-10$ in \eqref{eqn:quad_eq_matrix_opt_relax}. The error of a potential solution $\bx \in \R^n$ to \eqref{eqn:quad_eq} was measured as 
\begin{equation}
\label{eq:error_Quad}
    \text{error}(\bx) = \sum_{i=1}^{m} (\bx^\top \bm{Q}_i \bx - b_i)^2. 
\end{equation}

{\color{black}
We initialized Newton's method and Algorithm 1 in two ways. Table \ref{tab:my_table} displays the results when Newton's method was initialized randomly by sampling from a standard normal distribution and Algorithm 1 was initialized at a fixed diagonal matrix satisfying the constraint in \eqref{eqn:quad_eq_matrix_opt_relax}. 
Table \ref{tab:my_table_2} provides the results when Newton's method and Algorithm 1 were both initialized near a solution to \eqref{eqn:quad_eq}, that is, in all implementations of Newton's method each application of the method was initialized as
\begin{equation}\label{eqn:rand_int}
\bx^* + \eta \bm{\sigma},
\end{equation}
where $\bx^*$ was a solution to \eqref{eqn:quad_eq}, $\eta >0$ and each element of $\bm{\sigma}$ was sampled from a standard normal. 
Algorithm 1 was initialized at $\bx_0 \bx_0^\top$ where $\bx_0$ was generated by \eqref{eqn:rand_int}. 
We set $\eta = 0.4$ in our experiments. Note, Convex + Newton does not rely on initialization. 
}

The entries in each table provide the error of the approximated solutions obtained by the three methods above, and the errors obtained by solving the convex relaxation of \eqref{eqn:quad_eq_matrix_opt} and solving \eqref{eqn:quad_eq_matrix_opt_relax} with Algorithm \ref{alg:projected_grad} before application of Newton's method.  

\begin{table}[t]
\centering
\begin{tabular}{|cc|c|c|c|c|c|}
\hline
\multicolumn{2}{|l}{\textbf{Dimension ($n,m$)}}  & \multicolumn{1}{|c|}{\textbf{Newton}} & \multicolumn{1}{|c|}{\textbf{Convex}} & \multicolumn{1}{c}{\begin{tabular}[c]{@{}c@{}}\textbf {Convex}\\ \textbf{+Newton}\end{tabular}} & \multicolumn{1}{|c|}{\textbf{SCO}} & \multicolumn{1}{|c|}{\begin{tabular}[c]{@{}c@{}}\textbf{SCO}\\\textbf{+Newton}\end{tabular}} \\ \hline
\multicolumn{1}{|c|}{}                        & \textbf{Exp. 1} &   214.21  & 8.11e8 &  559.51  & 1.95  & {\bf 1.40}     \\ \cline{2-7} 
\multicolumn{1}{|c|}{}                        & \textbf{Exp. 2} &   331.68  & 6.30e8 &  385.27  & 14.44 & {\bf 3.80}     \\ \cline{2-7} 
\multicolumn{1}{|c|}{\textbf{(75,75)}}        & \textbf{Exp. 3} &   388.15  & 7.13e8 &  471.95  & 0.03  & {\bf 7.64e-11} \\ \cline{2-7} 
\multicolumn{1}{|c|}{}                        & \textbf{Exp. 4} &   69.09   & 6.76e8 &  570.88  & 37.95 & {\bf 6.16 }    \\ \cline{2-7} 
\multicolumn{1}{|c|}{}                        & \textbf{Exp. 5} &   154.01  & 1.03e9 &  141.76  & 54.30 & {\bf 7.37 }    \\ \hline \hline
\multicolumn{1}{|c|}{}                        & \textbf{Exp. 1} &   512.87  & 1.50e9 &  833.87  & 1.76  & {\bf 1.33 }    \\ \cline{2-7} 
\multicolumn{1}{|c|}{}                        & \textbf{Exp. 2} &   929.21  & 2.18e9 &  1.22e3  & 74.86 & {\bf 8.65 }    \\ \cline{2-7} 
\multicolumn{1}{|c|}{\textbf{(100,100)}}      & \textbf{Exp. 3} &   742.94  & 6.44e9 &  969.96  & 0.37  & {\bf 5.32e-13} \\ \cline{2-7} 
\multicolumn{1}{|c|}{}                        & \textbf{Exp. 4} &   660.32  & 2.05e9 &  386.52  & 2.25  & {\bf 3.06e-10} \\ \cline{2-7} 
\multicolumn{1}{|c|}{}                        & \textbf{Exp. 5} &   748.96  & 4.09e9 &  832.66  & 51.74 & {\bf 7.19}     \\ \hline
\end{tabular}
\caption{Errors of the three different solvers and the two initial solvers for different synthetic settings. The column titled ``Dimension" states the size of the matrices, $\bm{Q}_i \in \R^{n \times n}$, and the number of quadratic equations in \eqref{eqn:quad_eq}; the rows give the individual experiments for each dimension. 
The errors (according to \eqref{eq:error_Quad}) of the proposed methods are given below their respective headers. 
The column titled ``Convex" provides the error of the  
rank-1 projection of the solution to the convex relaxation of \eqref{eqn:quad_eq_matrix_opt};
the column ``SCO" states the error of the rank-1 projection of the approximate solution to \eqref{eqn:quad_eq_matrix_opt_relax} with $\delta = 1e-10$.
The bolded numbers mark the best error obtained for each experiment.}
\label{tab:my_table}
\end{table}

In Table \ref{tab:my_table}, we see
{SCO + Newton} significantly outperformed {Newton} and {Convex + Newton} in our experiments. 
The results show {SCO + Newton} located solutions in three of the ten tests while the other methods failed to compute any solutions to \eqref{eqn:quad_eq}. 
We note also the rank-1 projections of the approximated solutions to \eqref{eqn:quad_eq_matrix_opt_relax}, obtained by Algorithm \ref{alg:projected_grad}, given under header {``SCO"} in Table \ref{tab:my_table}, were always better than the approximated solutions of {Newton} and {Convex + Newton}. 
The vast discrepancy between the projected solutions of the convex relaxation, given under the header {``Convex"} in Table \ref{tab:my_table}, and \eqref{eqn:quad_eq_matrix_opt_relax} demonstrate the over-relaxed nature of removing the rank condition from \eqref{eqn:quad_eq_matrix_opt}. 
Though a global minimizer to \eqref{eqn:quad_eq_matrix_opt_relax} was not computed, the resulting approximate solution was significantly better than the global minimizer of the convex relaxation, by at least a factor of $10^6$. 
{\color{black}
In Table \ref{tab:my_table_2}, we observe SCO + Newton was able to obtain accurate solutions in nine of the ten experiments when being initialized near a solution to the system of equations. Newton, however, failed to locate any solutions while being initialized in the same neighborhood. We note further in Experiments 4 and 5 when $(n,m)=(75,75)$ that Algorithm \ref{alg:projected_grad} obtained a sufficiently accurate solution to \eqref{eqn:quad_eq} without applying Newton's method. These experiments showcase how our methodology effectively augments the local convergence of Newton's method.}
Thus, the framework offered by \eqref{eqn:SCO-Eig} presents a significantly better relaxation of the rank-1 condition which far surpasses the standard convex relaxation. This lends support for our framework being a means of relaxing rank constraints, and the results displayed Algorithm \ref{alg:projected_grad} can still perform well even when the constraint set is non-convex. 

\begin{table}[t]
\centering
\begin{tabular}{|cc|c|c|c|}
\hline
\multicolumn{2}{|l}{\textbf{Dimension ($n,m$)}}  & \multicolumn{1}{|c|}{\textbf{Newton}} & \multicolumn{1}{|c|}{\textbf{SCO}} & \multicolumn{1}{c|}{\begin{tabular}[c]{@{}c@{}}\textbf {SCO}\\ \textbf{+Newton}\end{tabular}} \\ \hline
\multicolumn{1}{|c|}{} & \textbf{Exp. 1} &   171.20  & 2.57 &  {\bf 2.62e-9}      \\ \cline{2-5} 
\multicolumn{1}{|c|}{} & \textbf{Exp. 2} &   68.98  & 0.14 &  {\bf 4.06e-13}       \\ \cline{2-5} 
\multicolumn{1}{|c|}{\textbf{(75,75)}}        & \textbf{Exp. 3} &   200.52  & 0.01 &  {\bf7.17e-9}   \\ \cline{2-5} 
\multicolumn{1}{|c|}{}                        & \textbf{Exp. 4} &   83.27   & 6.24e-5 &  {\bf2.40e-10}      \\ \cline{2-5} 
\multicolumn{1}{|c|}{}                        & \textbf{Exp. 5} &   93.94  & 1.24e-4 &  {\bf1.14e-9}      \\ \hline \hline
\multicolumn{1}{|c|}{}                        & \textbf{Exp. 1} &   347.66  & 0.16 &  {\bf 1.86e-10}      \\ \cline{2-5} 
\multicolumn{1}{|c|}{}                        & \textbf{Exp. 2} &   207.88  & {\bf0.01} &  0.09     \\ \cline{2-5} 
\multicolumn{1}{|c|}{\textbf{(100,100)}}      & \textbf{Exp. 3} &   151.35  & 0.59 &  {\bf 2.18e-13}   \\ \cline{2-5} 
\multicolumn{1}{|c|}{}                        & \textbf{Exp. 4} &   410.53  & 0.47 &  {\bf 2.14e-13}   \\ \cline{2-5} 
\multicolumn{1}{|c|}{}                        & \textbf{Exp. 5} &   193.35  & 0.12 &  {\bf2.05e-13}       \\ \hline
\end{tabular}
\caption{Errors of Newton, SCO + Newton,  and one initial solver for different synthetic settings with initialization near an optimal solution. The column titled ``Dimension" states the size of the matrices, $\bm{Q}_i \in \R^{n \times n}$, and the number of quadratic equations in \eqref{eqn:quad_eq}; the rows give the individual experiments for each dimension. 
The errors (according to \eqref{eq:error_Quad}) of the proposed methods are given below their respective headers.
The column ``SCO" states the error of the rank-1 projection of the approximate solution to \eqref{eqn:quad_eq_matrix_opt_relax} with $\delta = 1e-10$.
The bolded numbers mark the best error obtained for each experiment.}
\label{tab:my_table_2}
\end{table}

\section{Conclusion}\label{sec:conclusion}
This paper investigates the first matrix optimization model with linear inequality constrained eigenvalues. 
Theory was developed to understand the nature of the eigenvalue constraint set, and we presented KKT conditions for \eqref{eqn:SCO-Eig}. We additionally verified the accuracy of a relaxation which ensured the differentiability of the eigenvalue operator over the feasible domain. 
We proved global minima can be computed to \eqref{eqn:SCO-Eig} for linear objective problems, independent of the convexity of the constraint, and we proved how to compute exact projections. The computational complexity to obtain these global solutions is polynomial and only requires performing a spectral decomposition and solving a simple convex model, e.g., a linear program. 
Using these results, we developed two algorithms which compute first-order $\epsilon$-stationary points for \eqref{eqn:SCO-Eig} with a sublinear convergence rate. 
The practicality of our algorithms was accessed through numerical experimentation. Our example on solving systems of quadratic equations showcased a new and effective method for relaxing rank-constrained models.    

Many future directions are open for investigation. First, we plan on extending the framework of \eqref{eqn:SCO-Eig} to include equality constraints on the coordinates of the matrix,
\begin{align}\label{eqn:new_SCO-Eig}
\min&\; F(\bX) \\
\text{s.t.}&\; G_i(\bX) = 0, i = 1, \hdots, p \nonumber \\
           &\; \bA \lambda(\bX) \leq \bb \nonumber \\
           &\; \bX \in \Symn. \nonumber 
\end{align}
This paradigm will encompass most matrix optimization models studied over symmetric matrices and allow new spectral constraints never before considered. Additionally, we plan to develop approaches to solve models which constrain the singular values of non-square matrices and generalized singular values of tensors. As the applications of tensors expand in the burgeoning field of data science, the capability to manipulate the spectrum of these mathematical objects will become ever more important.

\section{Appendix}
\subsection{Proof for Section \ref{sec:constraint_set}}\label{sec:appen_1}
The convexity of $f_i$ in Theorem \ref{thm:convex_constraint} follows from the following proposition. 
\begin{proposition}\label{thm:f_convex}
If $\alpha_1 \geq \hdots \geq \alpha_n$, then $f(\bX) = \sum_{i=1}^{n} \alpha_i \lambda_i(\bX)$ is a convex function.
\end{proposition}
\begin{proof}
It is well-known $g_k(\bX) := \sum_{i=1}^{k}\lambda_i(\bX)$ and $h_p(\bX) := \sum_{j=0}^{p}\lambda_{n-j}(\bX)$ are convex and concave functions respectively for any $1 \leq k\leq n$ and $0\leq p\leq n-1$ \cite{boyd2004convex}. From this and simple operations which preserve convexity it follows $g_{\alpha}(\bX) := \sum_{i=1}^{n} \alpha_i \lambda_i(\bX)$ and $h_{\beta}(\bX) := \sum_{j=0}^{n-1}\beta_{n-j} \lambda_{n-j}(\bX)$ with $\alpha_1 \geq \hdots \geq \alpha_n \geq 0$ and $\beta_n \leq \hdots \leq \beta_1 \leq 0$ are convex. Thus,  
\[
g_{\alpha}(\bX) + h_{\beta}(\bX) = \sum_{i=1}^{n}(\alpha_i + \beta_i)\lambda_i(\bX)
\]
is convex. The equivalence of the sets 
$\left\{ \by + \bz \in \R^n \; | \; y_1 \geq \hdots \geq y_n \geq 0, \; z_n \leq \hdots \leq z_1 \leq 0 \right\}$
and 
$
\left\{ \bx \in \R^n \; | \; x_1 \geq \hdots \geq x_n \right\}
$ completes the argument. 
\end{proof}

\subsection{Proof of Theorem \ref{thm:projected_grad_Thm}}
We restate and prove the convergence result for Algorithm \ref{alg:projected_grad}. 
Our argument is given for a general vector space $\mathcal{V}$ with associated inner product $\langle \cdot, \cdot \rangle_{\cV}$ and induced norm $\|\cdot\|_{\cV}$. For the remainder of this section, we drop the subscript and refer to the inner product and norm on $\mathcal{V}$ as $\langle \cdot, \cdot \rangle$ and $\|\cdot\|$ respectively. 
Thus, the Frobenius norm used in the description of Algorithm \ref{alg:projected_grad} has been replaced with our general norm for $\cV$. 
We assume the norm associated with $\mathcal{V}$ defines the first-order $\epsilon$-stationary condition. 
Let $P_C^\delta$ denote the inexact projection onto a convex subset $\mathcal{C}$ of $\cV$. 
\begin{theorem}\label{thm:gen_projected_grad_thm}
Assume $F$ is gradient Lipschitz with parameter $L>0$, the initial level set, i.e., $\left\{ \bX \in \cC \;|\; F(\bX) \leq F(\bX_0) \right\}$, is a bounded subset of $\cV$ with diameter $D$, there exists $F^*$ such that $F^* \leq F(\bX)$ for all $\bX\in \cC$, there exists $M>0$ such that $\|\nabla F(\bX_k)\| \leq M$ for all $k\geq 0$, and $\cC$ is a convex subset of the normed vector space $\cV$. If inexact projections are computed with sufficient accuracy, dependent on $\epsilon$, then Algorithm \ref{alg:projected_grad} will converge to a first-order $\epsilon$-stationary point of $\min\{ F(\bX) \; | \; \bX \in \mathcal{C}\}$ in $\mathcal{O}(\epsilon^{-2})$ iterations. More explicitly, a first-order $\epsilon$-stationary point will be returned after no more than 
\[
  \left( \frac{4(D+Mh_{\text{low}}+1)^2}{h_{\text{low}}^2}\right) \left(\frac{F(\bX_0) - F^*}{\alpha}\right)\frac{1}{\epsilon^2}                                    \text{ iterations} 
\]
provided the accuracy of the inexact projections, $\delta$, satisfies $\delta \leq \min\left\{ \frac{h_{\text{low}}}{2}\epsilon, \; \frac{1}{2}\epsilon_{\text{tol}}^2\right\}$, where $h_{low} := \tau_1/(L+2\alpha)$ and 
$
\epsilon_{\text{tol}}:= \frac{h_{\text{low}} \epsilon}{2(D+Mh_{\text{low}}+1)}. 
$ 
\end{theorem}
%
\begin{proof}
Compute $\bX_{k+1} \in \text{P}_{ \mathcal{C}}^{\delta}(\bX_k  - h_k \nabla F(\bX_k))$. 
If $\|\bX_{k+1}-\bX_k\|\leq \epsilon_{\text{tol}}$, we shall prove $\bX_k$ is a first-order $\epsilon$-stationary point. 
If $\|\bX_{k+1}-\bX_k\| > \epsilon_{\text{tol}}$, we will show a sufficient decrease can be obtained provided the stepsize $h_k$ and precision $\delta>0$ are appropriately sized. 
To these ends, assume $\|\bX_{k+1}-\bX_k\| > \epsilon_{\text{tol}}$. 
If $h_k \leq (L+2\alpha)^{-1}$ and $\delta \leq \frac{1}{2} \epsilon_{\text{tol}}^2$, then it follows 
\begin{align}
&\langle \bX_k  - h_k \nabla F(\bX_k) - \bX_{k+1}, \bX_k - \bX_{k+1} \rangle \leq \delta  \nonumber \\
\implies &\langle h_k \nabla F(\bX_k), \bX_{k+1} - \bX_k \rangle \leq \delta - \|\bX_{k+1}-\bX_k\|^2 \nonumber \\
\implies & \langle \nabla F(\bX_k), \bX_{k+1}-\bX_k \rangle  \leq h_k^{-1} \left(\delta - \|\bX_{k+1}-\bX_k\|^2  \right). \nonumber 
\end{align}
By our assumptions, $\delta \leq \frac{1}{2}\epsilon_{\text{tol}}^2 \leq \frac{1}{2} \|\bX_{k+1} - \bX_k\|^2$. 
Therefore, 
\begin{align}
\langle \nabla F(\bX_k), \bX_{k+1}-\bX_k \rangle  \leq \frac{-h_k^{-1}}{2} \|\bX_{k+1}-\bX_k\|^2 &\leq  -\left(\frac{L+2\alpha}{2}\right)\|\bX_{k+1} - \bX_k\|^2. \nonumber 
\end{align}
Thus, from the gradient Lipschitz assumption on $F$ and the above inequality, 
\begin{align}\label{eqn:help_F_decr}
F(\bX_{k+1}) &\leq F(\bX_k) + \langle \nabla F(\bX_k), \bX_{k+1} - \bX_{k} \rangle + \frac{L}{2}\|\bX_{k+1} - \bX_k \|^2 \nonumber \\
&\leq F(\bX_k) - \left( -\frac{L}{2} + \alpha + \frac{L}{2} \right)\|\bX_{k+1} - \bX_k \|^2 \nonumber \\
&= F(\bX_k) - \alpha \|\bX_{k+1} - \bX_k \|^2.
\end{align}
Hence, the line-search shall terminate with a sufficient decrease obtained after no more than $\lceil \log_{2}\left( h(2\alpha + L) \right) \rceil$ inner iterations. 
So, for all $k \geq 0$ we have 
\[
\|\bX_{k+1} - \bX_k \|^2 \leq \frac{F(\bX_k) - F(\bX_{k+1})}{\alpha}.
\]
Summing this inequality up from $k=0$ to $k=K-1$ and using the lower bound $F^*$ we obtain
\begin{equation}\label{eqn:worst_case}
\min_{0 \leq k \leq K-1} \|\bX_{k+1} - \bX_k \|^2 \leq \left(\frac{F(\bX_0) - F^*}{\alpha}\right)\frac{1}{K}.
\end{equation}
We now prove if consecutive iterates generated by Algorithm \ref{alg:projected_grad} are sufficiently close together and the approximated projections are sufficiently accurate then a first-order $\epsilon$-stationary point has been obtained. 
By the Lipschitz gradient assumption on $F$ and the boundedness of the initial level set there exists $M>0$ such that $\|\nabla F(\bX_k)\| \leq M$ for all $k\geq 0$. 
Let $D$ be the diameter of the initial level set where the diameter of a set $S$ is defined as $\sup\left\{\|\bx - \by\| \; | \; \bx, \by \in S \right\}$. 
$D$ is finite due to the boundedness of the initial level. 
Also, since a sufficient decrease is obtained if $h_k \leq (L+2\alpha)^{-1}$, it follows $h_{\text{low}} = \tau_1(L+2\alpha)^{-1}$ lower bounds the stepsize $h_k$ for all $k\geq 0$.  
Assume $\|\bX_k - \bX_{k+1}\| \leq  \epsilon_{\text{tol}}$ and $\delta \leq \min\left\{\frac{h_{\text{low}}}{2}\epsilon, \; \frac{1}{2} \epsilon_{\text{tol}}^2 \right\}$, where 
\[
\epsilon_{\text{tol}}:= \frac{h_{\text{low}} \epsilon}{2(D+Mh_{\text{low}}+1)}. 
\]
Then for all $\bX \in \mathcal{C}$ such that $\|\bX - \bX_{k}\| \leq 1$ we have 
\begin{align}
&\langle \bX_k - h_k \nabla F(\bX_k) - \bX_{k+1}, \; \bX - \bX_{k+1}\rangle \leq \delta \nonumber \\
\implies &\langle - h_k \nabla F(\bX_k), \; \bX - \bX_{k+1}\rangle \leq \delta + \langle \bX_{k+1} - \bX_k,  \bX - \bX_{k+1} \rangle \nonumber \\
\implies &\langle - h_k \nabla F(\bX_k), \; \bX - \bX_{k+1}\rangle \leq \delta + \|\bX_{k+1} - \bX_k \|\cdot \| \bX - \bX_k + \bX_k - \bX_{k+1} \| \nonumber \\
\implies &\langle - h_k \nabla F(\bX_k), \; \bX - \bX_{k+1}\rangle \leq \delta + \|\bX_{k+1} - \bX_k \|\cdot \left(\|\bX - \bX_k\| + \| \bX_k - \bX_{k+1}\| \right) \nonumber \\ 
\implies &\langle - h_k \nabla F(\bX_k), \; \bX - \bX_{k+1}\rangle \leq \delta + (1+D)\|\bX_{k+1} - \bX_k \|\nonumber \\
\implies &\langle \nabla F(\bX_k), \bX - \bX_{k+1} \rangle \geq -\delta h_k^{-1} - (1+D)h_k^{-1} \|\bX_{k+1} - \bX_k\| \nonumber \\ 
\implies &\langle \nabla F(\bX_k), \bX - \bX_{k+1} \rangle \geq -\delta h_{\text{low}}^{-1} - (1+D)h_{\text{low}}^{-1} \|\bX_{k+1} - \bX_k\| \nonumber \\
\implies &\langle \nabla F(\bX_k), \bX - \bX_{k} \rangle = -\delta h_{\text{low}}^{-1} - (1+D)h_{\text{low}}^{-1} \|\bX_{k+1} - \bX_k\| - \langle \nabla F(\bX_k), \bX_k - \bX_{k+1} \rangle \nonumber \\
\implies &\langle \nabla F(\bX_k), \bX - \bX_{k} \rangle \geq -\delta h_{\text{low}}^{-1} - (1+D)h_{\text{low}}^{-1} \|\bX_{k+1} - \bX_k\| - M \|\bX_k - \bX_{k+1}\| \nonumber \\
\implies &\langle \nabla F(\bX_k), \bX - \bX_{k} \rangle = -\delta h_{\text{low}}^{-1} - ((1+D)h_{\text{low}}^{-1}+M) \|\bX_{k+1} - \bX_k\| \nonumber \\
\implies &\langle \nabla F(\bX_k), \bX - \bX_{k} \rangle \geq -\epsilon 
\end{align}
where the first implication comes from the definition of $\bX_{k+1} \in \text{P}_{ \mathcal{C}}^{\delta}(\bX_k  - h_k \nabla F(\bX_k))$, the sixth comes from $h_{\text{low}}$ lower bounding $h_k$ for all $k$, the eighth is a product of the bound on the norm of the gradient of $F$, and the last implication follows from the bounds on $\delta$ and $\|\bX_k - \bX_{k+1}\|$. 
Therefore, under these conditions $\bX_k$ is a first-order $\epsilon$-stationary point, and from \eqref{eqn:worst_case} a point satisfying $\| \bX_k - \bX_{k+1}\| \leq \epsilon_{\text{tol}}$ will be obtained within $K$ iterations provided 
\[
K \geq 
    \left( \frac{4(D+Mh_{\text{low}}+1)^2}{h_{\text{low}}^2}\right) \left(\frac{F(\bX_0) - F^*}{\alpha}\right)\cdot \frac{1}{\epsilon^2}.           
\]
\end{proof} 
\section*{Acknowledgments}
This material is based upon work supported by the National Science Foundation Graduate 
Research Fellowship Program under Grant No. 2237827 and NSF Award DMS-2152766. Any opinions, findings, 
and conclusions or recommendations expressed in this material are those of the author(s) 
and do not necessarily reflect the views of the National Science Foundation.

   

\end{document}